\documentclass[11pt]{article}
\usepackage{mathrsfs}
\usepackage{amsfonts}
\usepackage{xcolor}
\usepackage{extarrows}
\usepackage{indentfirst,amssymb,amsmath,graphicx,amsthm,amsfonts,mathrsfs,multicol,amscd}
\newcommand{\R}{\mathbb{R}}

\newtheorem{theorem}{Theorem}[section]
\newtheorem{corollary}{Corollary}

\newtheorem{lemma}[theorem]{Lemma}
\newtheorem{proposition}[theorem]{Proposition}

\theoremstyle{definition}

\newtheorem{remark}{Remark}

\def\ds{\displaystyle}

\begin{document}
\title{\LARGE\bf{On inequalities of Bliss-Moser type   with loss of compactness in $\R^N$  } }
\date{}
 \author{Yanyan Guo$^{1}$, Huxiao Luo$^{2}$$\thanks{{\small Corresponding author. E-mail: yanyangcx@126.com (Y. Guo),
 luohuxiao@zjnu.edu.cn (H. Luo), bernhard.ruf@unimi.it (B. Ruf).}}$, Bernhard Ruf
 $^{3}$\\
\small 1 School of Mathematics and Statistics, Central China Normal University, Wuhan 430079, China \\
\small 2 Department of Mathematics, Zhejiang Normal University, Jinhua, Zhejiang, 321004, China \\
\small 3 Istituto Lombardo - Accademia di Scienze e Lettere, 20121 Milan, Italy
}\maketitle

\begin{center}
		\begin{minipage}{13cm}
			\par
			\small  {\bf Abstract:}
   We prove the following {\it Limiting Bliss inequalities}
\begin{equation}\nonumber
\sup\limits_{v(0) = 0, \int_0^1|v'|^Ndx=1 }\int_0^1 e^{\beta\left(\log\frac{e}{s}\right)\frac{v^N(s)}{s^{N-1}}}ds\leq C(N,\beta), \ \hbox{ for } \beta \le 1
\end{equation}
The inequalities are optimal with respect to $\beta \le 1$; there is compactness for $\beta<1$, and along the {\it infinitesimal Moser sequence} for $\beta = 1$.

Moreover, we show that the improved inequalities
\begin{equation}\nonumber
\sup\limits_{v(0) = 0, \int_0^1|v'|^Ndx=1 }\int_0^1 e^{\left(\log\frac{e}{s}+\gamma\log\log\frac{e}{s}\right)\frac{v^N(s)}{s^{N-1}}}ds\leq C(N,\gamma)
\end{equation}
hold for $\gamma\leq1$, and for $\gamma=1$ the inequalities are critical with loss of compactness. The inequalities are optimal: no further improvement in the coefficient of the exponent is possible.

The second result extends the result in \cite{DRU} from $N=2$ to general dimensions $N\geq2$.

			\vskip2mm
			\par
			{\bf Keywords:} Trudinger-Moser inequalities, Bliss inequalities, Loss of compactness.
			\vskip2mm
			\par
			{\bf MSC(2010): } 35J20, 35J25, 35J50.
			
		\end{minipage}
	\end{center}

\section{Introduction}
\setcounter{equation}{0}
It is well-known that the Trudinger-Moser type inequalities can be considered as the
borderline case of the Sobolev embeddings. Indeed, since $$W^{1,N} (\R^N)\hookrightarrow L^q(\R^N),~\forall q\in [N,+\infty),~~ \text{but} ~W^{1,N} (\R^N)\not\hookrightarrow L^\infty(\R^N),$$ it has been shown by Trudinger \cite{Trudinger}
(see also Yudovich \cite{Yudovich} and Pohozaev \cite{Pohozaev}) that $W^{1,N}_0(\Omega)\hookrightarrow L_{\varphi_N}(\Omega)$, where $\Omega\subset\R^N$
is a domain with finite measure, $L_{\varphi_N}(\Omega)$ is the Orlicz space associated with the Young
function $\varphi_N(t)= e^{\alpha|t|^{\frac{N}{N-1}}} - 1$ for some $\alpha>0$.  Several years later, Moser \cite{Moser} was able to simplify
Trudinger's proof, and to determine the optimal threshold; for convenience, we state the result as follows:
\begin{proposition} (\cite{Moser}, Moser-Trudinger Inequality). There exists a positive
constant $C_0$ depending only on $N$ such that
\begin{equation}\label{MT}
\sup\limits_{u\in C^1_c(\Omega),\int_{\Omega}
|\nabla u|^N\leq1}
\int_{\Omega}
e^{\alpha|u|^{\frac{N}{N-1}}}
dx \leq C_0|\Omega|
\end{equation}
holds for all $\alpha \leq \alpha_N = N[\omega_{N-1}]^{\frac{1}{N-1}}$, where $\Omega$ is any bounded domain in $\R^N$.
Moreover, when $\alpha > \alpha_N$, the above supremum is always infinite.
\end{proposition}

In the course of his proof of inequality \eqref{MT}, J. Moser used symmetrization arguments,
by which it is sufficient to prove the Moser-Trudinger inequality for
$\Omega = B_R(0)$ and for functions $u \in W^{1,N}_0 (B_R(0))$ that are radial and monotonically
decreasing. A further change of variable
$$\frac{|x|^N}{R^N}=e^{-t},\quad w(t):=N^{N-\frac{1}{N}}[\omega_{N-1}]^{\frac{1}{N}}u\left(Re^{-\frac{t}{N}}\right)$$
converts the involved integrals as follows:
\begin{equation*}
 \aligned
&\int_{B_R(0)}|\nabla u|^Ndx=\int_0^{+\infty}|w'(t)|^Ndt,\\
&\frac{1}{|B_R(0)|}\int_{B_R(0)}e^{\alpha|u|^{\frac{N}{N-1}}}dx=\int_0^{+\infty}
e^{\beta|w|^{\frac{N}{N-1}}-t}dt,\endaligned
\end{equation*}
where $\beta = \frac{\alpha}{\alpha_N}$, and \eqref{MT} then can be established by proving the following one-dimensional
inequality.
\begin{proposition} \label{prop1.2}(\cite{Moser}). Let $2 \leq p < +\infty$, $0 < \beta\leq 1$ and $\frac{1}{p}+\frac{1}{q} = 1$, then there exists a constant
$C(p, \beta) > 0$ such that the inequality
\begin{equation}\label{20240917-e2}
\int_0^{+\infty}
e^{\beta|w|^{q}-t}dt \leq C(p, \beta)
\end{equation}
holds for all $C^1$ functions $w : [0,+\infty) \to [0, +\infty)$ satisfying $$w(0) = 0,~w' \geq 0~ \text{and} ~\int_0^{+\infty}
|w'(t)|^pdt \leq 1.$$
\end{proposition}
In addition, it was shown, surprisingly, by L. Carleson and S.-Y. A. Chang \cite{CSYAC}
that the best exponent $\beta = 1$ in \eqref{20240917-e2} is attained. This is in sharp contrast to the case of
the critical Sobolev embeddings where it is known that the best constant is never
attained on bounded domains.

Studies have also extended \eqref{MT} to unbounded domains, one can refer to the works cited in \cite{AT}, \cite{IM}, \cite{R1}. Furthermore, Trudinger-Moser type inequalities have been generalized to include weights in the exponential integral \cite{AS, CT, Tarsi}. Additionally, these Trudinger-Moser inequalities are closely related to elliptic equations with nonlinearities of critical Trudinger-Moser growth. Since the foundational work of Adimurthi \cite{AA} and subsequent research by de Figueiredo, Miyagaki, and Ruf \cite{FMR}, there has been extensive work on the existence of solutions for such equations. For more, refer to \cite{LL, O}, and the cited references within. The existence of solutions in the slightly supercritical range has also been explored \cite{R2, R3}. It is important to note that all these results concern inequalities and equations with critical Trudinger-Moser growth present a lack of compactness. In all cases, this non-compactness is due to a single point concentration which occurs at the origin in the radially symmetric case (respectively at infinity in the transformed Moser inequality of Proposition \ref{prop1.2}).

 In the work \cite{DRU}, do \'{O}, Ruf and Ubilla consider a new sharp one-dimensional integral inequality of Moser type with a variable coefficient $c(x)$ in the exponent which is singular at the origin. In this case, the non-compactness is due to a blow up of the gradient near the origin for sequences of functions which tend uniformly to zero. Some related  details on this phenomenon are provided below.

In the study of the Trudinger-Moser inequality, an important role  is played by the so-called
Moser-sequence, which is defined by the broken line functions
\begin{equation}\label{broken line functions}
v_j(t):=
\left\{
\begin{array}{ll}
\aligned
&j^{-\frac{1}{N}}t,\quad 0\leq t\leq j, \\
&j^{\frac{N-1}{N}},\quad t\geq j.
\endaligned
\end{array}
\right.
\end{equation}
As observed by J. Moser, these functions are maximizers of the basic inequality
\begin{equation}\label{20240917-e1}
y(t)\leq t^{1-\frac{1}{N}}\left(\int_0^{+\infty}|y'(\tau)|^Nd\tau\right)^{\frac{1}{N}},
\end{equation}
for all
$$
y\in \left\{u \in W^{1,N}(0, +\infty): u(0) = 0\right\}.$$
For $j \to+\infty$, the sequence of functions $v_j$ concentrates at $+\infty$, in the sense that
$v_j$ tend pointwise to zero in every $0 < j < +\infty$, while tending to infinity at $+\infty$.
Furthermore, $v'_j(t)$ is a vanishing sequence, since $v'_j(t)\to 0$ uniformly on $[0, +\infty)$.

The Moser sequence \eqref{broken line functions}
 serves to demonstrate the optimality of the inequality \eqref{20240917-e2}: for any $\beta>1$
 $$\int_0^{+\infty}
e^{\beta|v_j|^{\frac{N}{N-1}}-t}dt\geq \int_j^{+\infty}
e^{\beta j-t}dt=e^{j(\beta-1)}\to+\infty,~~\text{as}~j\to+\infty.$$
The Moser sequence \eqref{broken line functions} also shows the non-compactness of Moser's functional: one has
$$v_j\rightharpoonup 0 ~\text{in}~ W^{1,N}(0,+\infty) ~\text{while}~\int_0^{+\infty}e^{v_j^{\frac{N}{N-1}}-t}dt\to 2>1=\int_0^{+\infty}e^{-t}dt. $$

On the other hand, if we look at the Moser sequence \eqref{broken line functions} near zero, replacing $j$ by $\frac{1}{j}$, or more
precisely, considering
\begin{equation}\label{moser}
w_j(t):=
\left\{
\begin{array}{ll}
\aligned
&j^{\frac{1}{N}}t\ ,\quad 0\leq t\leq\frac{1}{j}, \\
&\frac 1 {j^{\frac{N-1}{N}}}\ ,\quad \frac{1}{j}\leq t\leq1,
\endaligned
\end{array}
\right.
\end{equation}
then $w_j$ again maximizes \eqref{20240917-e1} on $\{w \in H^1(0, 1), w(0) = 1\}$. However, the sequence $w_j$ exhibits dual characteristics in comparison to $v_j$, namely, $w_j$ is a vanishing sequence, converging uniformly to zero on the interval $[0, 1]$, while its derivative $w'_j(t)$ concentrates near $0$, tending pointwise to zero in the open interval $(0,1)$, and tending to $+\infty$ near the origin. We will call the sequence $w_j(t)$ the {\it infinitesimal Moser sequence}.

Now, similar to the Sobolev inequalities, we consider the Bliss inequalities \cite{Bliss}: $\forall N>1, k\geq1$ (if $k=1$, then it is the Hardy inequality \cite{Hardy})
\begin{equation}\label{Bliss}
	\begin{array}{ll}
\ds \int_0^1\left(\frac{u^N}{x^{N-1+\frac{1}{k}}}\right)^kdx\leq C_{N,k}\left(\int_0^1|u'|^Ndx\right)^k, \vspace{0.2cm} \\
 \hbox{ for all } \ u \in W^{1,N}(0,1) \ \hbox{ with } \ u(0) = 0 \ ,
 \end{array}
\end{equation}
where
\begin{equation}\label{CBliss}
C_{N,k}=\frac{1}{(N-1)k}\left[\frac{(k-1)\Gamma(\frac{Nk}{k-1})}{\Gamma(\frac{1}{k-1})\Gamma(\frac{Nk-1}{k-1})}\right]^{k-1}.
\end{equation}
By \eqref{Bliss}, using the estimate for the Bliss embedding constants $C_{N,k}$ and the Taylor expansion of the exponential function, we first prove a $N-$ dimensional  {\it Limiting-Bliss inequality} as follows
$$
\sup\limits_{w(0)=0, \int_0^1 |w'(s)|^Nds=1}\int_0^1 e^{\beta\,\log\frac{e}{s}\,\frac{w^N(s)}{s^{N-1}}}ds\leq C(N, \beta) \ , \ \hbox{ for } \ \beta < 1.
$$

Moreover, we prove that this Limiting-Bliss inequality has compactness for the infinitesimal Moser sequence \eqref{moser}.
Let
$$E_N:= \left\{u \in W^{1,N}(0, 1): u(0) = 0, \int_0^1 |u'|^Ndx=1\right\}.$$
\begin{theorem}\label{Th1}
For $N\geq2$, let
\begin{equation}\nonumber
I_\beta(v)=\int_0^1 e^{\beta\,\log\frac{e}{s}\,\frac{v^N(s)}{s^{N-1}}}ds.
\end{equation}
Then
\begin{equation}\label{supI}
\sup\limits_{v\in E_N}I_\beta(v)<\infty\Longleftrightarrow \beta\leq1.
\end{equation}
\noindent
The exponent $\beta=1$ is optimal; for $\beta < 1$ the functioal $I_\beta$ is weakly continuous  on $E_N$, and for $\beta = 1$ the functional $I_1$ is weakly continuous along the infinitesimal Moser sequence \eqref{moser}.
More precisely, we show
\vspace{-0.2cm}
\begin{equation}\nonumber
    I_\beta (w_j)\to
    \left\{
    \begin{array}{lcl}
         +\infty,\ \ &\hbox{if}\ \beta>1,\\
         1=I_1(0),\ \ &\hbox{if}\ \beta=1,
    \end{array}
    \right.
\end{equation}
and for any sequence  $\{u_n\}\subset E_N$  with $u_n \rightharpoonup u $
$$
\hspace{-1cm} I_\beta (u_n) \to I_\beta(u) \ , \ \ \qquad \hbox{if}\ \beta <  1.
$$
\end{theorem}
\par \bigskip

Since the Limiting-Bliss inequality does not show non-compactness for the borderline case $\beta = 1$, it is not critical in this sense.

We  mention that in \cite{DRU}, do \'{O}, Ruf and Ubilla give an optimal $\log\log$ improvement of the limiting Bliss inequality in $E_2$, for which the infinitesimal Moser sequence $w_j(t)$, as given in \eqref{moser}, plays an analogous role to the Moser sequence in the optimal TM-inequality. That is:\par
- to distinguish between subcriticality and supercriticality;
\par
 - to show non-compactness for the critical case.
\par \medskip
Inspired by \cite{DRU}, we give a $\log\log$ type Bliss-Moser inequality in $E_N$, which is critical with loss of compactness. To be specific, we extend the result in \cite{DRU} from $N = 2$ to general dimensions $N \geq 2$.

\begin{theorem}\label{Th2}
For $N\geq2$, let
\begin{equation}\nonumber
J_\gamma(v)=\int_0^1 e^{\left(\log\frac{e}{s}+\gamma\log\log\frac{e}{s}\right)\frac{v^N(s)}{s^{N-1}}}ds.
\end{equation}
Then
\begin{equation}\nonumber
\sup\limits_{v\in E_N}J_\gamma(v)<\infty\Longleftrightarrow \gamma\leq1.
\end{equation}
For $\gamma=1$, the inequality is critical with loss of compactness: the functional $J_1$ fails to be weakly continuous along the  infinitesimal Moser sequence
\eqref{moser}.
More precisely,
\begin{equation}\nonumber
    J_\gamma (w_j)\to
    \left\{
    \begin{array}{lcl}
         +\infty,\ \ &\hbox{if}\ \gamma>1,\\
         c\ge e+1>1=J_1(0),\ \ &\hbox{if}\ \gamma=1,\\
         1=J_\gamma(0),\ \ &\hbox{if}\ \gamma<1.
    \end{array}
    \right.
\end{equation}
\end{theorem}
\begin{remark}
The non-compactness of $J_\gamma$ for $\gamma = 1$ occurs, as in analogous situations with single point blow up, due to an {\it asymptotic invariance} of $J_1$ for the infinitesimal Moser sequence, as can be seen in the proof of Proposition \ref{p3.1} below.
\end{remark}
\begin{corollary}\label{cor}
	We remark that $J_1$ is optimal with respect to any perturbation $h(\frac 1s)$ with $h(\frac 1s) \to \infty$, as $s \to 0$, in the exponent. More precisely, let $h : [1,\infty)\to \mathbb R^+$ continuous, increasing, with $h(t) \to +\infty$ as $t \to 0$, and $\lim_{t\to \infty}\frac {h(t)}{\log\log et} = 0$. Then
	$$\sup_{v \in E_N} J_{1,h}(v) := \sup_{v \in E_N}\int_0^1e^{(\log \frac es + \log\log \frac es + h(\frac 1s))\frac{v^N}{s^{N-1}}} = \infty.
	$$
Example: $h(\frac 1s) = \log\log\log \frac {e^e}s$.
\end{corollary}

\section{The proof of Theorem \ref{Th1} }
\setcounter{equation}{0}
\noindent
First, we give an estimate for the Bliss embedding constants $C_{N,k}$ in \eqref{CBliss}.
\begin{proposition}\label{pro1} For any fixed $N>1$,
\begin{equation}\nonumber
k\,C_{N,k}\to \frac{1}{N-1}\,e^{1+\frac{1}{2}+\cdot\cdot\cdot \frac{1}{N-1}}=:C_N\ ,\quad\text{as}~k\to\infty
\end{equation}
\end{proposition}
\begin{proof}
Let $m=k-1$, by \eqref{CBliss}
\begin{equation*}
\begin{array}{lll}
&&(N-1)kC_{N,k}\vspace{0.3cm}\\
&=&\ds \bigg[\frac{m\Gamma(N+\frac{N}{m})}{\Gamma(\frac{1}{m})\Gamma(N+\frac{N-1}{m})}\bigg]^{m}
\vspace{0.3cm}\\

&=&\ds \bigg[\frac{(N-1+\frac{N}{m})(N-2+\frac{N}{m})\cdot\cdot\cdot(1+\frac{N}{m})\Gamma(1+\frac{N}{m})}{\Gamma(\frac{1}{m}+1)(N-1+\frac{N-1}{m})(N-2+\frac{N-1}{m})\cdot\cdot\cdot(1+\frac{N-1}{m})\Gamma(1+\frac{N-1}{m})}\bigg]^{m}\vspace{0.3cm} \\
&=&\ds \bigg[\Big(1+\frac{1}{(N-1)m+(N-1)}\Big)\Big(1+\frac{1}{(N-2)m+(N-1)}\Big)\cdot\cdot\cdot\Big(1+\frac{1}{m+(N-1)}\Big)\bigg]^{m}\vspace{0.3cm} \\
&&\ds \cdot\bigg[\frac{\Gamma(1+\frac{N}{m})}{\Gamma(1+\frac{1}{m})\Gamma(1+\frac{N-1}{m})} \bigg]^{m} \ ,
\end{array}
\end{equation*}
since
\begin{equation*}
\begin{array}{ll}
&\ds \left[\Big(1+\frac{1}{(N-1)m+(N-1)}\Big)\Big(1+\frac{1}{(N-2)m+(N-1)}\Big)\cdot\cdot\cdot\Big(1+\frac{1}{m+(N-1)}\Big)\right]^{m}\vspace{0.5cm} \\
& \longrightarrow\quad \ds e^{\frac{1}{N-1}}e^{\frac{1}{N-2}}\cdot\cdot\cdot e^{\frac{1}{N-(N-1)}},\quad\text{as}~m\to\infty\ ,
\end{array}
\end{equation*}
and
\begin{equation*}
\begin{array}{ll}
&\ds\bigg[\frac{\Gamma(1+\frac{N}{m})}{\Gamma(1+\frac{1}{m})\Gamma(1+\frac{N-1}{m})} \bigg]^{m}
\vspace{0.3cm}\\ =&\ds\bigg[\frac{\Gamma(1)+\Gamma'(1)\frac{N}{m}+O(\frac{1}{m^2})}{\left(\Gamma(1)+\Gamma'(1)\frac{1}{m}+O(\frac{1}{m^2})\right)\left(\Gamma(1)+\Gamma'(1)\frac{N-1}{m}+O(\frac{1}{m^2})\right)} \bigg]^{m}\vspace{0.3cm}\\
=&\ds\bigg[\Big(1+\Gamma'(1)\frac{N}{m}+O(\frac{1}{m^2})\Big)\Big(1-\Gamma'(1)\frac{1}{m}+O(\frac{1}{m^2})\Big)\Big(1-\Gamma'(1)\frac{N-1}{m}+O(\frac{1}{m^2})\Big) \bigg]^{m}\vspace{0.3cm}\\
=&\ds\bigg[\Big(1+\Gamma'(1)\frac{N}{m}+O(\frac{1}{m^2})\Big)\Big(1-\Gamma'(1)\frac{N}{m}+O(\frac{1}{m^2})\Big)\bigg]^{m}\vspace{0.3cm}\\
=&\ds\bigg[1+O(\frac{1}{m^2})\bigg]^{m}\to1,\quad\text{as}~m\to\infty,
\end{array}
\end{equation*}
we get the conclusion.
\end{proof}

\begin{remark}
    In the ground-breaking result of L. Carleson and S.-Y. A. Chang \cite{CSYAC}, the number
    $1+e^{1+\frac{1}{2}+\frac{1}{3}+\cdot\cdot\cdot+\frac{1}{N-1}}$ is an energy threshold that plays a crucial role in proving the existence of the extremal function of the Moser inequality $$\sup\limits_{w\in E_N}\int_0^{+\infty}
e^{\beta|w|^{\frac{N}{N-1}}-t}dt.$$
The connection between the constant $C_N$ defined in Proposition \ref{pro1} and the number
    $1+e^{1+\frac{1}{2}+\frac{1}{3}+\cdot\cdot\cdot+\frac{1}{N-1}}$ seems to imply an correlation between the Limiting-Bliss inequality and the Moser inequality.
\end{remark}
Trudinger \cite{Trudinger} (see also Yudovich \cite{Yudovich}, Pohozaev \cite{Pohozaev}) proved his inequality by using Taylor expansion and then applying the Sobolev inequalities to each term. We follow this idea, using the Bliss inequalities in place of the Sobolev inequalities, to prove the following Limiting Bliss inequality:
\begin{proposition}\label{0918th1}
The supremum
\begin{equation}\label{supremum}
\sup\limits_{v\in E_N}\int_0^1 e^{\beta\left(\log\frac{e}{s}\right)\frac{v^N(s)}{s^{N-1}}}ds
\end{equation}
 is finite for $\beta < 1$.
\end{proposition}
\begin{proof}
 It is sufficient to consider the case $\beta\in(0,1)$.
According to Taylor expansion, we have
\begin{equation}\label{20240924-e1}
\aligned
\int_0^1 e^{\beta\left(\log\frac{e}{s}\right)\frac{v^N(s)}{s^{N-1}}}ds
=& 1+\int_0^1\sum_{k\geq1}\frac{1}{k!} \left[\beta\left(\log\frac{e}{s}\right)\frac{v^N(s)}{s^{N-1}}\right]^kds\\
=& 1+\int_0^1\sum_{k\geq1}\frac{1}{k!} \left[\beta\left(s^{\frac{1}{k}}\log\frac{e}{s}\right)\frac{v^N(s)}{s^{N-1+\frac{1}{k}}}\right]^kds.
\endaligned
\end{equation}
Note that
\begin{equation}\label{20240705-e1}
\max_{s>0}\big\{s^{\frac{1}{k}}\log\frac{e}{s}\big\}=\frac{k}{e^{1-\frac{1}{k}}}.
\end{equation}
Hence we can estimate the  sum in \eqref{20240924-e1} by using \eqref{20240705-e1} and the Bliss inequalities \eqref{Bliss}
\begin{equation}\nonumber
\aligned
 &\int_0^1\sum_{k\geq1}\frac{1}{k!} \left[\beta\left(s^{\frac{1}{k}}\log\frac{e}{s}\right)\frac{v^N(s)}{s^{N-1+\frac{1}{k}}}\right]^kds \\
 \leq& \ \sum_{k\geq1}\frac{1}{k!} \left(\beta k\right)^ke^{1-k}\int_0^1\left[\frac{v^N(s)}{s^{N-1+\frac{1}{k}}}\right]^kds\\
 \leq&\  e\sum_{k\geq1}\frac{k^{k-1}}{k!} \left(\frac{\beta}{e}\right)^{k}kC_{N,k}\left[\int_0^1|v'|^Nds\right]^k\\
 =&\ e\sum_{k\geq1}\frac{k^{k-1}}{k!} \left(\frac{\beta}{e}\right)^{k}(C_N+o(1)).
\endaligned
\end{equation}
Applying the ratio test, we get
\begin{equation}\nonumber
\aligned
\frac{\frac{(k+1)^{k}}{(k+1)!} \left(\frac{\beta}{e}\right)^{k+1}}
{\frac{k^{k-1}}{k!} \left(\frac{\beta}{e}\right)^{k}}=\frac{\beta}{e}\frac{(k+1)^{k-1}}{k^{k-1}}=\frac{\beta}{e}\left(1+\frac{1}{k}\right)^{k-1}\to\beta.
\endaligned
\end{equation}
Hence, if $\beta< 1$, then the series converges. This proves the proposition.
\end{proof}

Next, we
prove that the supremum \eqref{supremum} is infinite for $\beta>1$ by using the infinitesimal Moser sequence
\eqref{moser}.

\begin{proposition}\label{0918th2}
The supremum \eqref{supremum} is infinite for $\beta>1$.
\end{proposition}
\begin{proof}
We use the infinitesimal Moser sequence \eqref{moser} to demonstrate the sharpness of the Limiting-Bliss inequality \eqref{supremum}
with respect to $\beta=1$. Indeed, setting $\beta=1+\delta$, for $\delta>0$ and $N\ge2$ we have
\begin{equation}\nonumber
\aligned
\int_0^1 e^{\beta\left(\log\frac{e}{s}\right)\frac{w_j^N(s)}{s^{N-1}}}ds
\geq & \int_0^{\frac{1}{j}} e^{(1+\delta) js\log(\frac{e}{s})}ds
\vspace{0.3cm}\\
\geq &\int_0^{\frac{1}{j}} e^{(1+\delta) js\log(je)} ds
\vspace{0.3cm}\\
=&\int_0^{1} e^{(1+\delta) t\log(je)} \frac{1}{j}dt
\vspace{0.3cm}\\
=&\frac{1}{j}\frac{1}{(1+\delta) \log(je)}e^{(1+\delta) \log(je)}-\frac{1}{j}\frac{1}{(1+\delta) \log(je)}
\vspace{0.3cm}\\
=&\frac{e^{(1+\delta) }}{(1+\delta) }\frac{j^{\delta}}{\log(je)}-\frac{1}{j}\frac{1}{(1+\delta) \log(je)}
\vspace{0.3cm}\\
\to&+\infty,\quad\text{as}~j\to+\infty.
\endaligned
\end{equation} \end{proof}
\par
The case $\beta = 1$ cannot be proved with the "Taylor series method". But we observe that
$$
\sup\limits_{v\in E_N}I_1(v)=\sup\limits_{v\in E_N}\int_0^1 e^{\left(\log\frac{e}{s}\right)\frac{v^N(s)}{s^{N-1}}}ds \le
\sup\limits_{v\in E_N}\int_0^1 e^{\left(\log\frac{e}{s}+\gamma\log\log\frac{e}{s}\right)\frac{v^N(s)}{s^{N-1}}}ds< \infty
$$
follows from  Theorem \ref{Th2}, for $0 \le\gamma \le 1$; we therefore omit the proof here.

\par \bigskip
To complete the proof of Theorem \ref{Th1}, all that remains to prove is
\begin{proposition}
    Let $\{w_j\}$ be the infinitesimal Moser sequence defined in
\eqref{moser}.
Then
\begin{equation}\nonumber
    I_\beta (w_j)\to
    \left\{
    \begin{array}{lcl}
         +\infty,\ \ &\hbox{if}\ \beta>1,\\
         1=I_1(0),\ \ &\hbox{if}\ \beta=1,
    \end{array}
    \right.
\end{equation}
and (for subsequence)
$$I_\beta(u_n) \to I_\beta(u)\ , \ \hbox{ for } \  u_n \rightharpoonup u \ \hbox{ in } \ E_N,\ \hbox{ if }  \beta < 1$$
\end{proposition}
\begin{proof}
{\bf Case $\beta>1$.}
Using \eqref{moser}
\begin{equation*}
\aligned
I_\beta(w_j)>\int_0^{\frac{1}{j}}\big(\,\frac{e}{s}\,\big)^{\beta js}ds
\geq\int_0^{\frac{1}{j}}\left(ej\right)^{\beta js}ds
=\frac{(ej)^\beta}{\beta j \log(ej)}-\frac{1}{\beta j \log(ej)}\to+\infty.
\endaligned
\end{equation*}
{\bf Case $\beta =1$. }
\par \smallskip \noindent
{\it Upper estimate.} \
We write for any fixed $\epsilon>0$, using \eqref{moser}
\begin{equation}\nonumber
\aligned
I_1(w_j)
=&\int_0^{\frac{1}{(1+\epsilon)j}}\left(\frac{e}{s}\right)^{ js}ds+\int_{\frac{1}{(1+\epsilon)j}}^{\frac{1}{j}}\left(\frac{e}{s}\right)^{ js}ds+\int_{\frac{1}{j}}^{\frac{1+\epsilon}{j}}\left(\frac{e}{s}\right)^{\frac{1}{(js)^{N-1}}}ds
+\int_{\frac{1+\epsilon}{j}}^{1}\left(\frac{e}{s}\right)^{\frac{1}{(js)^{N-1}}}ds\\
:=& A_j+B_j+C_j+D_j,
\endaligned
\end{equation}
and estimate
\begin{equation}\nonumber
\aligned
A_j\le& \int_0^{\frac{1}{(1+\epsilon)j}}\left(\frac{e}{s}\right)^{\frac{1}{1+\epsilon}}ds=\frac{e^{\frac{1}{1+\epsilon}}}{1-\frac{1}{1+\epsilon}}\Big(\frac{1}{(1+\epsilon)j}\Big)^{1-\frac{1}{1+\epsilon}}\to 0, \\
B_j\leq& \int^{\frac{1}{j}}_{\frac{1}{(1+\epsilon)j}}\frac{e}{s}\, ds=e\log{(1+\epsilon)},
\\
C_j\leq& \int_{\frac{1}{j}}^{\frac{1+\epsilon}{j}}\frac{e}{s}\, ds=e\log{(1+\epsilon)}.
\endaligned
\end{equation}
Next, we show that
\begin{equation}\nonumber
    D_j=\displaystyle \int_{\frac{1+\epsilon}{j}}^{1} \Big(\frac{e}{s}\Big)^{\frac{1}{(js)^{N-1}}}\, ds\to 1 \ \hbox{as}\ j\to \infty.
\end{equation}
Indeed, we write
\begin{equation}\nonumber
    \displaystyle \int_{\frac{1+\epsilon}{j}}^{1} \Big(\frac{e}{s}\Big)^{\frac{1}{(js)^{N-1}}}\, ds=\int_0^1\chi_{[\frac{1+\epsilon}{j}, 1]}\Big(\frac{e}{s}\Big)^{\frac{1}{(js)^{N-1}}}\, ds
\end{equation}
and note that on $[\frac{1+\epsilon}{j}, 1]$ we have $\frac{1}{(js)^{N-1}}\le \frac{1}{(1+\epsilon)^{N-1}}$ and
\begin{equation}\label{eq:57}\nonumber
\chi_{[\frac{1+\epsilon}{j}, 1]}\Big(\frac{e}{s}\Big)^{\frac{1}{(js)^{N-1}}}\le \chi_{[\frac{1+\epsilon}{j}, 1]}\Big(\frac{e}{s}\Big)^{\frac{1}{(1+\epsilon)^{N-1}}}\in L^1(0,1).
\end{equation}
Furthermore,
\begin{equation}\nonumber
\chi_{[\frac{1+\epsilon}{j}, 1]}\Big(\frac{e}{s}\Big)^{\frac{1}{(js)^{N-1}}}\to 1 \ \hbox{as}\ j \to \infty
\end{equation}
pointwise in $s$. Hence, by the Dominated Lebesgue theorem we conclude $D_j\to1$. Combining the estimates of $A_j, B_j, C_j$ and $D_j$, for any $\epsilon>0$ we have
$$I_1(w_j)\leq 1+2e\log(1+\epsilon)+o(1),\quad\text{as}~j\to\infty.$$
Then by the arbitrariness of $\epsilon$, we get  $\limsup_{j\to\infty} I_1(w_j)\leq1.$

\par \smallskip \noindent
{\it Lower estimate.} \ Since $\beta \left(\log\frac{e}{s}\right)\frac{w_j^N(s)}{s^{N-1}} \ge 0,\  s\in (0,1)$, we have
$$I_\beta(w_j) = \int_0^1 e^{\beta\left(\log\frac{e}{s}\right)\frac{w_j^N(s)}{s^{N-1}}}ds \ge 1  \ , \ \forall \ j.
$$
Combining Upper estimate and Lower estimate, we get $I_1(w_j)\to1.$
\par \medskip \noindent
{\bf Case $\beta<1$.}\ For $u_n \rightharpoonup u$ in $E_N$ we have $u_n \to u$ in any $L^p(0,1)$, and hence $u_n(x) \to u(x), a.e.$ (subsequence). Furthermore, from inequality \eqref{basic} below we have
$$u(s) \le s^{\frac {N-1}N} \ , \ s \in (0,1)$$
and hence
$$e^{\beta (\log \frac es ) \frac {u(s)^N}{s^{N-1}} }\le e^{\beta (\log \frac es)} = \big(\frac es\big)^\beta \in L^1(0,1)  ,  \ \hbox{ for } \ \beta < 1\ .
$$
Then the claim follows by Lebesgue dominated convergence.
\end{proof}

\section{Approximate Moser functions}
\noindent
For any $w\in E_N$ and $s>0$, we see from
\begin{equation}\label{basic}
w(s)=\int_0^{s}w'(t)dt\leq s^{\frac{N-1}{N}}\left(\int_0^{s}|w'(t)|^Ndt\right)^{\frac{1}{N}}
\end{equation}
that
\begin{equation}\label{20240709-e1}
	\frac{w^N(s)}{s^{N-1}}\leq1.
\end{equation}
Inspired by \cite{Moser}, we prove that if for some $w \in E_N$ the value
$$\max\limits_{s\in[0,1]}\frac{w^N(s)}{s^{N-1}} \quad \hbox{ is close to }\  1 \ ,$$
 then $w\in E_N$ is close to one of the infinitesimal Moser functions \eqref{moser}.
\begin{lemma}\label{Key} For any $w\in E_N$, let
	\begin{equation}\label{20240706-e1}
		1-\delta=\max\limits_{s\in[0,1]}\frac{w^N(s)}{s^{N-1}}=\frac{w^N(a)}{a^{N-1}},~~(\delta, a\in[0,1]~\text{depend~on~} w),
	\end{equation}
	then
	\begin{equation}\label{Close}
		\Big(\frac{w(a)}{a}\Big)^{N-2}\int_0^{a}\Big|w'(s)-\frac{w(a)}{a}\Big|^2ds+\int_{a}^1|w'(s)|^Nds\leq\delta.
	\end{equation}
\end{lemma}
\begin{proof}
	First, by \eqref{20240709-e1} we have
	\begin{equation}\label{20240706-e2}
		\frac{w^N(a)}{a^{N-1}}\leq \int_0^{a}|w'(s)|^Nds.
	\end{equation}
	Then by the H\"{o}lder inequality and \eqref{20240706-e2},
	\begin{equation}\nonumber
		\aligned
		&\left(\frac{w(a)}{a}\right)^{N-2}\int_0^{a}\left|w'(s)-\frac{w(a)}{a}\right|^2ds+\int_{a}^1|w'(s)|^Nds\\
		=&\int_0^{a}\left(\frac{w(a)}{a}\right)^{N-2}\left|w'(s)\right|^2ds-\frac{w(a)^N}{a^{N-1}}+1-\int_0^{a}|w'(s)|^Nds\\
		=&\int_0^{a}\left(\frac{w(a)}{a}\right)^{N-2}\left|w'(s)\right|^2ds-\int_0^{a}|w'(s)|^Nds+\delta\\
		\leq&\left[\int_0^{a}\left(\frac{w(a)}{a}\right)^{N}ds\right]^{\frac{N-2}{N}}\left[\int_0^{a}|w'(s)|^Nds\right]^{\frac{2}{N}}-\int_0^{a}|w'(s)|^Nds+\delta\\
		=&\bigg[\frac{w(a)^N}{a^{N-1}}\bigg]^{\frac{N-2}{N}}\left[\int_0^{a}|w'(s)|^Nds\right]^{\frac{2}{N}}-\int_0^{a}|w'(s)|^Nds+\delta\\
		\leq&\left[\int_0^{a}|w'(s)|^Nds\right]^{\frac{N-2}{N}}\left[\int_0^{a}|w'(s)|^Nds\right]^{\frac{2}{N}}-\int_0^{a}|w'(s)|^Nds+\delta\\
		=&\ \delta.
		\endaligned
	\end{equation}
\end{proof}
\begin{lemma}\label{Key1}
	Let $\delta$, $w$ and $a$ as in Lemma \ref{Key}. Then
	$$w(s)\leq w(a)+(s-a)^{1-\frac{1}{N}}\delta^{\frac{1}{N}},\quad\forall s\in[a,1].$$
\end{lemma}
\begin{proof}
	By the H\"{o}lder inequality,
	$$w(s)-w(a)=\int_{a}^s w'(t)dt\leq(s-a)^{1-\frac{1}{N}}\left(\int_{a}^1|w'(t)|^Ndt\right)^{\frac{1}{N}}.$$
	Then using \eqref{Close}, we get the conclusion.
\end{proof}

\begin{lemma}\label{Key2}
	Let $\delta$, $w$ and $a$ as in Lemma \ref{Key}. Then
	$$w(s)\leq s\frac{w(a)}{a}+(a-s)^{\frac{1}{2}}\delta^{\frac{1}{2}}\left(\frac{w(a)}{a}\right)^{-\frac{N-2}{2}},\quad\forall s\in[0,a].$$
\end{lemma}
\begin{proof}
	By the H\"{o}lder inequality,
	\begin{equation}\nonumber
		\aligned
		w(s)= &\ w(a)-\int^{a}_s w'(t)dt
		= s\frac{w(a)}{a}+\int_s^{a}\left(\frac{1}{a}w(a)-w'(t)\right)dt \\
		\leq&\ s\frac{w(a)}{a}+(a-s)^{\frac{1}{2}}\bigg(\int^{a}_s\Big|\frac{w(a)}{a}-w'(t)\Big|^2dt\bigg)^{\frac{1}{2}} .
		\endaligned
	\end{equation}
	Then using \eqref{Close}, we get the conclusion.
\end{proof}

\begin{lemma}\label{Key3}
	Let $w$ and $a$ as in Lemma \ref{Key}. If $0<\delta<\frac{1}{2}$, then
	$$w(s)\leq s\frac{1}{a^{\frac{1}{N}}}+s^{\frac{N-1}{N}}(2\delta)^{\frac{1}{N}},\quad\forall s\in[0,a].$$
\end{lemma}
\begin{proof}
	By \eqref{20240706-e1} we have
	$$w_1:= w(a) = (1- \delta)^{\frac{1}{N}}a^{\frac{N-1}{N}}.$$
	For fixed $\sigma\in (0, a)$ let
	\begin{equation}\label{20240708-e1}
		w^*:= \max\left\{w(\sigma) : w(0) = 0, w(a) = w_1, \int_0^{a}
		|w'(s)|^Nds \leq 1\right\}.
	\end{equation}
	The maximizer is the broken line function	\begin{equation}\nonumber
		\bar{w}(s)=	\left\{\begin{array}{rcl}
			&\frac{w^*}{\sigma}s,\quad &\forall s \in [0,\sigma],  \\
			&w^*+\frac{w_1-w^*}{a-\sigma}(s-\sigma),\quad &\forall s\in[\sigma,a].
		\end{array}\right.
	\end{equation}
	The integral condition in \eqref{20240708-e1} yields
	\begin{equation}\label{20240708-e2}
		\left(\frac{w^*}{\sigma}\right)^N\sigma+\left(\frac{w_1-w^*}{a-\sigma}\right)^N(a-\sigma)\leq1.
	\end{equation}
	The line connecting the origin with $(a, w_1)$ is given by
	$$\tilde{w}(s)=\left(\frac{1-\delta}{a}\right)^{\frac{1}{N}}s,$$
	and since clearly
	$$w^* \geq \tilde{w}(\sigma)$$
	we can write
	$$w^* =\left(\frac{1-\delta}{a}\right)^{\frac{1}{N}}
	(\sigma + \rho)$$ for some $\rho > 0$.
	The integral condition in \eqref{20240708-e1} then becomes
	$$\frac{\left|\left(\frac{1-\delta}{a}\right)^{\frac{1}{N}}(\sigma+\rho)\right|^N}{\sigma^N}\ \sigma+\frac{\left|(1-\delta)^{\frac{1}{N}}a^{\frac{N-1}{N}}-\left(\frac{1-\delta}{a}\right)^{\frac{1}{N}}(\sigma+\rho)\right|^N}{|a-\sigma|^N}(a-\sigma)\leq1,$$
	that is
	$$\left|1+\frac{\rho}{\sigma}\right|^N\sigma+\left|1-\frac{\rho}{a-\sigma}\right|^N(a-\sigma)\leq\frac{a}{(1-\delta)}.$$
	Then by $|1-x|^N\geq 1-Nx$ for all $x\in\R$,
	\begin{equation}\nonumber
		\aligned
		\left[1+N\frac{\rho}{\sigma}+\left(\frac{\rho}{\sigma}\right)^N\right]\sigma+\left[1-N\frac{\rho}{a-\sigma}\right](a-\sigma)
		\leq\frac{a}{1-\delta}.
		\endaligned
	\end{equation}
	That is
	\begin{equation}\label{20241026-e1}\nonumber
		\aligned
		a+\rho^{N}\frac{1}{\sigma^{N-1}}
		\leq\frac{a}{1-\delta}
		\endaligned
	\end{equation}
	and then
	$$\rho^N\leq\frac{\delta}{1-\delta}a\sigma^{N-1}.$$
	Note that $\frac{1}{1-\delta}<2$ by the assumption $0<\delta<\frac{1}{2}$,
	we get
	\begin{equation}\nonumber
		\rho<(2\delta a)^{\frac{1}{N}}\sigma^{\frac{N-1}{N}}.
	\end{equation}
	Therefore
	\begin{equation}\nonumber
		\bar{w}(\sigma)=w^*=\left(\frac{1-\delta}{a}\right)^{\frac{1}{N}}(\sigma+\rho)\leq\frac{1}{a^{\frac{1}{N}}}\sigma+(2\delta)^{\frac{1}{N}}\sigma^{\frac{N-1}{N}}.
	\end{equation}
\end{proof}

\section{The proof of Theorem \ref{Th2}}
\setcounter{equation}{0}

\begin{proposition}\label{p3.1}
    Let
\begin{equation}\label{J}
J_\gamma(v):=\int_0^1 e^{\left(\log\frac{e}{s}+\gamma\log\log\frac{e}{s}\right)\frac{v^N(s)}{s^{N-1}}}ds=\int_0^1 \big(\frac{e}{s}\ (\log \frac{e}{s})^\gamma\big)^{\frac{v^N(s)}{s^{N-1}}}ds
\end{equation}
and $\{w_j\}$ be the infinitesimal Moser sequence \eqref{moser}.
Then
\begin{equation}\nonumber
    J_\gamma (w_j)\to
    \left\{
    \begin{array}{lcl}
         +\infty,\ \ &\hbox{if}\ \gamma>1,\\
         c\ge e+1>1=J_1(0),\ \ &\hbox{if}\ \gamma=1,\\
         1=J_\gamma(0),\ \ &\hbox{if}\ \gamma<1.
    \end{array}
    \right.
\end{equation}
\end{proposition}

\begin{proof}
    i) $\gamma>1:$ We can estimate, using that \ $\frac{e\;(\log(\frac{e}{s}))^\gamma}{s}$ is decreasing in $s$
\begin{equation}\nonumber
\aligned
     J_\gamma (w_j(s))
    &\ge \displaystyle \int_{\frac{1}{2j}}^{\frac{1}{j}} \Big(\frac{e(\log \frac{e}{s})^\gamma}{s}\Big)^{\frac{w_j^N(s)}{s^{N-1}}}\, ds \vspace{4mm}\\
    &\ge \displaystyle \int_{\frac{1}{2j}}^{\frac{1}{j}} \Big( e j\big(\log (ej)\big)^\gamma\Big)^{js}\, ds
\xlongequal{js=t} \displaystyle \int_{\frac{1}{2}}^1 \Big( e j\big(\log (ej)\big)^\gamma\Big)^t \frac{1}{j}\, dt
    \vspace{4mm}\\
    &= \displaystyle \int_{\frac{1}{2}}^1 e^{t\log \big(ej(\log (ej))^\gamma\big)} \frac{1}{j}\, dt
    \vspace{4mm}\\
    &= \displaystyle \frac{1}{j}\frac{ej(\log (ej))^\gamma-\big(ej(\log (ej))^\gamma\big)^{\frac{1}{2}}}{\log \big(ej(\log (ej))^\gamma\big)}
    \vspace{4mm}\\
    &= \displaystyle \frac{e(\log (ej))^\gamma\left(1-\frac{1}{\big(ej(\log (ej))^\gamma\big)^{\frac{1}{2}}}\right)}{\log (ej)\left(1+\frac{\log \left(\left(\log (ej)\right)^\gamma\right)}{\log (ej)}\right)}
    \vspace{4mm}\\
    &= \displaystyle \frac{e(\log (ej))^\gamma(1+o(1))}{\log (ej)(1+o(1))} \vspace{4mm}\\
    &\ge e(\log(ej))^{\gamma-1}(1+o(1))
    \to \displaystyle  \infty\  \hbox{as}\ j\to \infty.
\endaligned
\end{equation}
\par \smallskip \noindent
ii) $\gamma=1$, {\bf upper estimate:} Using \eqref{moser}, we calculate
\begin{equation}\nonumber
\begin{array}{lcl}
     J_1(w_j(s))&=&\displaystyle \int_0^{\frac{1}{2j}}(\frac{e\log \frac{e}{s}}{s})^{\frac{js^N}{s^{N-1}}}\, ds+\int_{\frac{1}{2j}}^{\frac{1}{j}} (\frac{e\log \frac{e}{s}}{s})^{\frac{js^N}{s^{N-1}}}
     \vspace{3mm}\\
     &&\displaystyle +\int_{\frac{1}{j}}^{\frac{2}{j}} \Big(\frac{e\log \frac{e}{s}}{s}\Big)^{\frac{1}{j^{N-1}s^{N-1}}}\, ds+\int_{\frac{2}{j}}^{1} \Big(\frac{e\log \frac{e}{s}}{s}\Big)^{\frac{1}{j^{N-1}s^{N-1}}}\, ds
     \vspace{3mm}\\
     &=:&\displaystyle A_j+B_j+C_j+D_j
\end{array}
\end{equation}
and estimate
\begin{equation}\nonumber
    \begin{array}{lcl}
         A_j&=&\displaystyle \int_0^{\frac{1}{2j}}(\frac{e\log \frac{e}{s}}{s})^{\frac{js^N}{s^{N-1}}}\, ds= \int_0^{\frac{1}{2j}}(\frac{e+\log \frac{1}{s^e}}{s})^{js}\, ds
         \vspace{3mm}\\
         &\le&\displaystyle \int_0^{\frac{1}{2j}}(\frac{(\frac{c}{s})^{1/4}+c}{s})^{js}\, ds\le \int_0^{\frac{1}{2j}}(\frac{c}{s^{\frac{5}{4}}}+\frac{c}{s})^{1/2}\, ds
         \vspace{3mm}\\
         &\le&\displaystyle \int_0^{\frac{1}{2j}}\frac{c_1}{s^{\frac{5}{8}}}\, ds\to 0,
    \end{array}
\end{equation}

\begin{equation}\nonumber
    \begin{array}{lcl}
            B_j&=& \displaystyle \int_{\frac{1}{2j}}^{\frac{1}{j}} (\frac{e\log \frac{e}{s}}{s})^{\frac{js^N}{s^{N-1}}}\, ds\le \int_{\frac{1}{2j}}^{\frac{1}{j}} (2e j\log (2ej))^{js}\, ds
            \vspace{3mm}\\
            &\xlongequal{js=t}& \displaystyle \int_{\frac{1}{2}}^{1} (2e j\log (2ej))^t\frac{1}{j}\, dt=\int_{\frac{1}{2}}^{1} e^{t \log (2e j\log (2ej))}\frac{1}{j}\, dt
            \vspace{3mm}\\
            &=& \displaystyle \frac{e^{t \log (2e j\log (2ej))}}{\log (2e j\log (2ej))} \frac{1}{j}\bigg|^1_{\frac{1}{2}}=\frac{2e j\log (2ej)-(2e j\log (2ej))^{1/2}}{j \log (2e j\log (2ej))}
            \vspace{3mm}\\
            &=& \displaystyle \frac{2e j\log (2ej)(1-(2e j\log (2ej))^{-1/2})}{j \log (2ej)(1+\frac{\log (\log(2ej))}{\log (2ej)})}
            \vspace{3mm}\\
            &=& \displaystyle \frac{2e \log (2ej)(1+o(1))}{\log (2ej)(1+o(1))}
            \vspace{3mm}\\
            &=& \displaystyle 2e(1+o(1))\le c,
    \end{array}
\end{equation}
\par \bigskip
\begin{equation}\nonumber
    \begin{array}{lcl}
         C_j&=&\displaystyle \int_{\frac{1}{j}}^{\frac{2}{j}} \Big(\frac{e\log \frac{e}{s}}{s}\Big)^{\frac{1}{j^{N-1}s^{N-1}}}\, ds\le  \int_{\frac{1}{j}}^{\frac{2}{j}} \Big(ej\log (ej) \Big)^{\frac{1}{j^{N-1}s^{N-1}}}\, ds
         \vspace{3mm}\\
         &\xlongequal{js=1/t}&\displaystyle \int^{\frac{1}{2}}_1 \Big(ej\log (ej) \Big)^{t^{N-1}}\frac{1}{j}\frac{-1}{t^2}\,dt\le 4\int_{\frac{1}{2}}^1 \Big(ej\log (ej) \Big)^{t^{N-1}}\frac{1}{j}\,dt
          \vspace{3mm}\\
          &\le & \displaystyle 4\int_{\frac{1}{2}}^1 \Big(ej\log (ej) \Big)^{t}\frac{1}{j}\,dt=4\int_{\frac{1}{2}}^1 e^{t \log (ej\log (ej))}\frac{1}{j}\,dt
          \vspace{3mm}\\
          &= & \displaystyle \frac{4}{j}\frac{e^{t \log (e j\log (ej))}}{\log (e j\log (ej))} \bigg|^1_{\frac{1}{2}}=\frac{4e j\log (ej)(1+o(1))}{j \log (ej)(1+o(1))}
          \vspace{3mm}\\
          &= & \displaystyle 4e(1+o(1))\le c,
    \end{array}
\end{equation}

\begin{equation}\nonumber
    \begin{array}{lcl}
         D_j&=&\displaystyle \int_{\frac{2}{j}}^{1} \Big(\frac{e\log \frac{e}{s}}{s}\Big)^{\frac{1}{j^{N-1}s^{N-1}}}\, ds\le \int_{\frac{2}{j}}^{1} \Big(\frac{(\frac{c}{s})^{1/4}+c}{s}\Big)^{\frac{1}{j^{N-1}s^{N-1}}}\, ds
         \vspace{3mm}\\
          &\le & \displaystyle \int_{\frac{2}{j}}^{1} \Big(\frac{c}{s^{5/4}}+\frac{c}{s}\Big)^{\frac{1}{2^{N-1}}}\, ds\le \int_{\frac{2}{j}}^{1} \Big(\frac{c}{s^{5/4}}+\frac{c}{s}\Big)^{\frac{1}{2}}\, ds
          \vspace{3mm}\\
          &\le & \displaystyle \int_{\frac{2}{j}}^{1} \frac{c_1}{s^{5/8}}\, ds\le c.
    \end{array}
\end{equation}
\par
\bigskip \noindent
$\gamma=1$, {\bf lower estimate:} In this estimate the {\it asymptotic invariance} of $J_1(w_j)$ becomes visible: in the line (*) we see that the growth terms $j\log (ej)$ in the numerator and denominator cancel.
\begin{equation}\nonumber
    \begin{array}{ccl}
         J_1(w_j(s))&\ge& \displaystyle \int_{\frac{1}{2j}}^{\frac{1}{j}} \Big(\frac{e\log \frac{e}{s}}{s}\Big)^{js}\, ds+\int_{\frac{1}{j}}^{1} \Big(\frac{e\log \frac{e}{s}}{s}\Big)^{\frac{1}{j^{N-1}s^{N-1}}}\, ds
         \vspace{3mm}\\
          &= & \displaystyle  \int_{\frac{1}{2j}}^{\frac{1}{j}} e^{\big(\log \frac{e}{s}+\log(\log \frac{e}{s})\big)js}\, ds+\int_{\frac{1}{j}}^{1} e^{\big(\log \frac{e}{s}+\log(\log \frac{e}{s})\big)\frac{1}{j^{N-1}s^{N-1}}}\,ds
          \vspace{3mm}\\
          &\ge & \displaystyle \int_{\frac{1}{2j}}^{\frac{1}{j}} e^{\big(\log (ej)+\log(\log (ej))\big)js}\, ds+1-\frac{1}{j}
           \vspace{3mm}\\
          &\xlongequal{js=t} & \displaystyle \int_{\frac{1}{2}}^{1} e^{\big(\log (ej)+\log(\log (ej))\big)t}\frac{1}{j}\, dt+1-\frac{1}{j}
          \vspace{3mm}\\
          &=& \displaystyle \frac{1}{j}\frac{e^{\big(\log (ej)+\log(\log (ej))\big)t}}{\log (ej)+\log(\log (ej))}\bigg|_{\frac{1}{2}}^1+1-\frac{1}{j}
          \vspace{3mm}\\
          &=& \displaystyle \frac{1}{j} \frac{ej\log (ej)-(ej\log (ej))^{1/2}}{\log (ej)(1+\frac{\log (\log (ej))}{\log (ej)})}+1-\frac{1}{j}
          \vspace{3mm}\\
          (*)\hspace{-1cm} &=& \displaystyle  \frac{e\, j\log (ej)(1-\frac{(ej\log (ej))^{1/2}}{ej\log (ej)})}{j\log (ej)(1+\frac{\log (\log (ej))}{\log (ej)})}+1-\frac{1}{j}
          \vspace{3mm}\\
          &=& \displaystyle e(1+o(1))+1-\frac{1}{j}
          \vspace{3mm}\\
          &=& \displaystyle e+1+o(1).
    \end{array}
\end{equation}
\par \smallskip \noindent
iii) $\gamma<1:$ Following the upper estimate in ii) we have
\begin{equation}\nonumber
    \begin{array}{lcl}    J_\gamma(w_j(s))&=&\displaystyle \int_0^{\frac{1}{2j}}\Big(\frac{e(\log \frac{e}{s})^\gamma}{s}\Big)^{js}\, ds+\int_{\frac{1}{2j}}^{\frac{1}{j}} \Big(\frac{e(\log \frac{e}{s})^\gamma}{s}\Big)^{js}
     \vspace{3mm}\\
     &&\displaystyle +\int_{\frac{1}{j}}^{\frac{2}{j}} \Big(\frac{e(\log \frac{e}{s})^\gamma}{s}\Big)^{\frac{1}{j^{N-1}s^{N-1}}}\, ds+\int_{\frac{2}{j}}^{1} \Big(\frac{e(\log \frac{e}{s})^\gamma}{s}\Big)^{\frac{1}{j^{N-1}s^{N-1}}}\, ds
     \vspace{3mm}\\
     &=:&\displaystyle A_j+B_j+C_j+D_j.
    \end{array}
\end{equation}
$A_j$ goes to zero as above, and we show that $B_j$ and $C_j$ go to zero, while $D_j\to 1$. In fact, we have
\begin{equation}\nonumber
    \begin{array}{lcl}
         B_j&=&\displaystyle \int_{\frac{1}{2j}}^{\frac{1}{j}} (\frac{e(\log \frac{e}{s})^\gamma}{s})^{js}\le \int_{\frac{1}{2j}}^{\frac{1}{j}} (2e j(\log (2ej))^\gamma)^{js}\, ds
            \vspace{3mm}\\
            &\xlongequal{js=t}& \displaystyle \int_{\frac{1}{2}}^{1} (2e j(\log (2ej))^\gamma)^t\frac{1}{j}\, dt=\int_{\frac{1}{2}}^{1} e^{t \log (2e j(\log (2ej))^\gamma)}\frac{1}{j}\, dt
            \vspace{3mm}\\
            &=& \displaystyle \frac{e^{t \log (2e j(\log (2ej))^\gamma)}}{\log (2e j(\log (2ej))^\gamma)} \frac{1}{j}\bigg|^1_{\frac{1}{2}}=\frac{2e j(\log (2ej))^\gamma-(2e j(\log (2ej))^\gamma)^{1/2}}{ \log (2e j)+\log(\log (2ej))^\gamma}\frac{1}{j}
            \vspace{3mm}\\
            &=& \displaystyle \frac{2e j(\log (2ej))^\gamma(1-(2e j(\log (2ej))^\gamma)^{-1/2})}{j \log (2ej)(1+\frac{\log ((\log(2ej))^\gamma)}{\log (2ej)})}
            \vspace{3mm}\\
            &=& \displaystyle \frac{2e (\log (2ej))^\gamma(1+o(1))}{\log (2ej)(1+o(1))}
            \vspace{3mm}\\
            &=& \displaystyle 2e(\log(2ej))^{\gamma-1}(1+o(1))\to 0, \ \ \hbox{since $\gamma<1$},
    \end{array}
\end{equation}

\begin{equation}\nonumber
    \begin{array}{lcl}
         C_j&=&\displaystyle \int_{\frac{1}{j}}^{\frac{2}{j}} \Big(\frac{e(\log \frac{e}{s})^\gamma}{s}\Big)^{\frac{1}{j^{N-1}s^{N-1}}}\, ds\le  \int_{\frac{1}{j}}^{\frac{2}{j}} \Big(ej(\log (ej))^\gamma \Big)^{\frac{1}{j^{N-1}s^{N-1}}}\, ds
         \vspace{3mm}\\
         &\xlongequal{js=1/t}&\displaystyle \int^{\frac{1}{2}}_1 \Big(ej(\log (ej) )^\gamma \Big)^{t^{N-1}}\frac{1}{j}\frac{-1}{t^2}\,dt\le 4\int_{\frac{1}{2}}^1 \Big(ej(\log (ej))^\gamma \Big)^{t^{N-1}}\frac{1}{j}\,dt
          \vspace{3mm}\\
          &\le & \displaystyle 4\int_{\frac{1}{2}}^1 \Big(ej(\log (ej))^\gamma \Big)^{t}\frac{1}{j}\,dt=4\int_{\frac{1}{2}}^1 e^{t \log (ej(\log (ej))^\gamma)}\frac{1}{j}\,dt
          \vspace{3mm}\\
          &= & \displaystyle \frac{4}{j}\frac{e^{t \log (e j(\log (ej))^\gamma)}}{\log (e j(\log (ej))^\gamma)} \bigg|^1_{\frac{1}{2}}=\frac{4e j(\log (ej))^\gamma(1+o(1))}{j \log (ej)(1+o(1))}\to 0.
    \end{array}
\end{equation}

Next, we show that
\begin{equation}\nonumber
    D_j=\displaystyle \int_{\frac{2}{j}}^{1} \Big(\frac{e(\log \frac{e}{s})^\gamma}{s}\Big)^{\frac{1}{j^{N-1}s^{N-1}}}\, ds\to 1 \ \hbox{as}\ j\to \infty.
\end{equation}
Indeed, we write
\begin{equation}\label{eq:81}\nonumber
    \displaystyle \int_{\frac{2}{j}}^{1} \Big(\frac{e(\log \frac{e}{s})^\gamma}{s}\Big)^{\frac{1}{j^{N-1}s^{N-1}}}\, ds=\int_0^1\chi_{[\frac{2}{j}, 1]}\Big(\frac{e(\log \frac{e}{s})^\gamma}{s}\Big)^{\frac{1}{j^{N-1}s^{N-1}}}\, ds
\end{equation}
and note that on $[\frac{2}{j}, 1]$ we have $\frac{1}{(js)^{N-1}}\le \frac{1}{2^{N-1}}$ and
\begin{equation}\nonumber
    \chi_{[\frac{2}{j}, 1]}\Big(\frac{e(\log \frac{e}{s})^\gamma}{s}\Big)^{\frac{1}{j^{N-1}s^{N-1}}}\le \chi_{[\frac{2}{j}, 1]}\Big(\frac{e(\log \frac{e}{s})^\gamma}{s}\Big)^{\frac{1}{2^{N-1}}}\in L^1(0,1).
\end{equation}
Furthermore,
\begin{equation}\nonumber
    \chi_{[\frac{2}{j}, 1]}\Big(\frac{e(\log \frac{e}{s})^\gamma}{s}\Big)^{\frac{1}{j^{N-1}s^{N-1}}}=\chi_{[\frac{2}{j}, 1]}e^{\frac{1}{(js)^{N-1}}\big(\log\frac{e}{s}+\log((\log\frac{e}{s})^\gamma)\big)}\to 1 \ \hbox{as}\ j \to \infty
\end{equation}
pointwise in $s$. Hence, we conclude by the Dominated Lebesgue theorem.
\end{proof}
We insert here the
\begin{proof}[Proof of Corollary \ref{cor}:]
 indeed, it is sufficient to show that
$$ J_{1,h}(w_j) \to +\infty \ \hbox{ as }\  j \to \infty .$$
We follow the estimate above for:
$\gamma=1$, {\bf lower estimate}
\begin{equation}\nonumber
	\begin{array}{lcl}
		J_{1,h}(w_j(s))
		&\ge & \displaystyle \int_{\frac{1}{2j}}^{\frac{1}{j}} e^{\big(\log (ej)+\log(\log (ej))+ h(j)\big)js}\, ds
		\vspace{3mm}\\
		&\xlongequal{js=t} & \displaystyle \int_{\frac{1}{2}}^{1} e^{\big(\log (ej)+\log(\log (ej))+h(j)\big)t}\,\frac{1}{j}\, dt
		\vspace{3mm}\\
		&=& \displaystyle \frac{1}{j}\, \frac{e^{\big(\log (ej)+\log(\log (ej))+h(j)\big)t}}{\log (ej)+\log(\log (ej))+h(j)}\bigg|_{\frac{1}{2}}^1
		\vspace{3mm}\\
		&=& \displaystyle \frac{1}{j} \frac{ej\log (ej)e^{h(j)}-\big(ej\log (ej)e^{h(j)}\big)^{1/2}}{\log (ej)\big(1+\frac{\log (\log (ej))+ h(j)}{\log (ej)}\big)}
		\vspace{3mm}\\
		&=& \displaystyle  \frac{e\log (ej)e^{h(j)}\big(1-\frac{(ej\log (ej)e^{h(j)})^{1/2}}{ej\log (ej)e^{h(j)}}\big)}{\log (ej)(1+o(1))}
		\vspace{3mm}\\
		&=& \displaystyle e\, e^{h(j)}(1+o(1)) \to \infty \ , \ \hbox{ as } \ j \to \infty.
	\end{array}
\end{equation}
\end{proof}

\begin{proof}[Proof of Theorem\,\ref{Th2}]
From Proposition\,\ref{p3.1}, we have proved that if $\gamma>1$, then $$\sup\limits_{v\in E_N}J_\gamma(v)=+\infty.$$ Thus, it is enough to show that if $\gamma=1$, then
$$\sup\limits_{\int_0^1 |v'|^N=1,v(0)=0}J_\gamma(v)<\infty.$$
We assume now that $\gamma=1$. We also consider three cases.
\par \smallskip \noindent
{\bf Case a)} Suppose there exists $\delta_0 > 0$ such that $\delta=\delta(v)\geq \delta_0~ \forall~v\in E_N$, i.e.
$$1-\delta(v)=\max\limits_{s\in(0,1]}
\frac{|v(s)|^N}{s^{N-1}} = \frac{|v(a)|^N}{a^{N-1}}
\leq 1 - \delta_0~~\forall~v\in E_N. $$
Then we are done, since then
\begin{equation}\nonumber
    \int_0^1 \left(\frac{e\log\frac{e}{s}}{s}\right)^{\frac{v^N(s)}{s^{N-1}}}ds \leq \int_0^1 \left(\frac{e\log\frac{e}{s}}{s}\right)^{1-\delta_0}ds\leq c.
\end{equation}
\par \smallskip \noindent
{\bf Case b)} Next suppose that there exists a fixed $a_0 > 0$ such that the maximum point $a = a(v) \geq a_0$ for all $v\in E_N$.
Then we are also done: indeed, we then have by \eqref{20240709-e1}
\begin{equation}\nonumber
    \int_{\frac{a_0}{2^N}}^1\left(\frac{e\log\frac{e}{s}}{s}\right)^{\frac{v^N(s)}{s^{N-1}}}ds\leq \int_{\frac{a_0}{2^N}}^1\frac{e\log\frac{e}{s}}{s}ds\le\frac{e(2^N-a_0)}{a_0}\log(\frac{e\,2^N}{a_0}),
\end{equation}
while in the integral
$$\int_0^{\frac{a_0}{2^N}}\left(\frac{e\log\frac{e}{s}}{s}\right)^{\frac{v^N(s)}{s^{N-1}}}ds$$
we use Lemma \ref{Key3}
to get that for $s<\frac{a_0}{2^N}$ and $\delta\leq \frac{1}{2^{N+2}}$
(for all $v\in\{v\in E_N: \delta=\delta(v)>\frac{1}{2^{N+2}}\}$ then the above integral is finite by the same argument of Case a)
\begin{equation*}
\aligned
\frac{v(s)^N}{s^{N-1}}\leq&\frac{1}{s^{N-1}}\bigg[s\frac{1}{a_0^{\frac{1}{N}}}+s^{\frac{N-1}{N}}(2\delta)^{\frac{1}{N}}\bigg]^N\\
\leq&\frac{1}{s^{N-1}}2^{N-1}\left[\frac{s^N}{a_0}+s^{N-1}(2\delta)\right]\\
<&\frac{1}{2}+2^{N}\delta\leq\frac{3}{4}
\endaligned
\end{equation*}
and then
$$\int_0^{\frac{a_0}{2^N}}\left(\frac{e\log\frac{e}{s}}{s}\right)^{\frac{v^N(s)}{s^{N-1}}}ds\leq \int_0^{\frac{a_0}{2^N}}\left(\frac{e\log\frac{e}{s}}{s}\right)^{\frac{3}{4}} ds\leq c.
$$
\par \medskip \noindent
{\bf Case c)} Suppose that there exist sequences $v_j\in E_N$, $a_j \to 0$ and $\delta_j \to 0$ such that
\begin{equation}\label{eq:87}
    \max\limits_{s\in(0,1]}
\frac{|v_j(s)|^N}{s^{N-1}} = \frac{|v_j(a_j)|^N}{a_j^{N-1}}
= 1 - \delta_j.
\end{equation}
First note that by \eqref{20240709-e1},
\begin{equation}\label{eq:88}\nonumber
\aligned
    &\int_0^1 e^{\big(\log\frac{e\log\frac{e}{s}}{s}\big)\frac{|v_j(s)|^N}{s^{N-1}}}ds
    =\int_0^1 \Big(\frac{e\log\frac{e}{s}}{s}\Big)^{\frac{|v_j(s)|^N}{s^{N-1}}}ds \\
    =&\int_0^1 e^{\frac{|v_j(s)|^N}{s^{N-1}}} \Big(\frac{\log\frac{e}{s}}{s}\Big)^{\frac{|v_j(s)|^N}{s^{N-1}}}ds \\
    \le &\ e\int_0^1 \Big(\frac{\log\frac{e}{s}}{s}\Big)^{\frac{|v_j(s)|^N}{s^{N-1}}}ds.
    \endaligned
\end{equation}
To estimate the integral above, we divide the interval into two parts and discuss each interval in two cases:

1. The interval $(0, a_j):$ We distinguish two cases:
\par \smallskip \noindent
{\it Case 1:} there exists $d_0>0$ such that $\delta_j\log\frac{1}{a_j}\ge d_0$, for all $j$.
\par \smallskip \noindent
{\it Case 2:} it holds $\delta_j\log\frac{1}{a_j}\to 0$.

 2. The interval $(a_j, 1):$
We distinguish again the two cases:
\par \smallskip \noindent
{\it Case 1}: there exists $d_0 > 0$ such that $\delta_j\log\frac{1}{a_j}\ge d_0$, for all $j$.
\par \smallskip \noindent
{\it Case 2}: it holds $\delta_j\log\frac{1}{a_j}\to 0$.
\vspace{3mm}
\par \medskip \indent
{\bf 1. The interval $(0, a_j).$}
\par \smallskip \noindent
{\it Case 1:} We now select a fixed large $c>0$ and a fixed small $\nu>0$ which will be determined later, and first consider \\
(i) the left sub-interval $(0, (1-\nu)a_j)$; \\
(ii) the right sub-interval $I_j^0:=\big[a_j(1-c\delta_j), a_j\big]$.

For the middle part of the interval $(0, a_j)$, we divide it into the following union of small intervals: \\
(iii) $\bigcup_{k=1}^{k_j(\nu)-1}I_j^k$, where $k_j(\nu)\in \mathbb{N}$ denote the smallest integer such that $k_j(\nu)c\delta_j\ge \nu$, and $$
    I_j^k=[a_j(1-(k+1)c\delta_j), a_j(1-kc\delta_j)].
$$

(i) The interval $(0, (1-\nu)a_j)$. To estimate $\int_0^{(1-\nu)a_j} \Big(\frac{\log\frac{e}{s}}{s}\Big)^{\frac{|v_j(s)|^N}{s^{N-1}}}ds$, we first
use Lemma\,\ref{Key3} to show that for $s\in (0, (1-\nu)a_j)$
\begin{equation}\label{eq:15}\nonumber
    \begin{array}{lcl}
         \displaystyle &&\frac{|v_j(s)|^N}{s^{N-1}}
         \vspace{3mm}\\
         &\le& \displaystyle \frac{1}{s^{N-1}}\Big|s\frac{1}{a_j^{\frac{1}{N}}}+s^{\frac{N-1}{N}}(2\delta_j)^{\frac{1}{N}}\Big|^N
         \vspace{3mm}\\
         &=& \displaystyle \frac{1}{s^{N-1}}\Big|\frac{s^N}{a_j}+Ns^{\frac{N-1}{N}}(2\delta_j)^{\frac{1}{N}}\Big(\frac{s}{a_j^{\frac{1}{N}}}\Big)^{N-1}+\text{\footnotesize $\frac{N(N-1)}{2}$}s^{\frac{2(N-1)}{N}}(2\delta_j)^{\frac{2}{N}}\frac{s^{N-2}}{a_j^{\frac{N-2}{N}}}
         +\cdots+s^{N-1}(2\delta_j)\Big|
         \vspace{3mm}\\
         &=& \displaystyle \frac{s}{a_j}+N(\frac{s}{a_j})^{\frac{N-1}{N}}(2\delta_j)^{\frac{1}{N}}+\text{\footnotesize$\frac{N(N-1)}{2}$}(\frac{s}{a_j})^{\frac{N-2}{N}}(2\delta_j)^{\frac{2}{N}}+\cdots+2\delta_j
         \vspace{3mm}\\
         &\le& \displaystyle (1-\nu)+N(1-\nu)^{\frac{N-1}{N}}(2\delta_j)^{\frac{1}{N}}+\text{\footnotesize $\frac{N(N-1)}{2}$}(1-\nu)^{\frac{N-2}{N}}(2\delta_j)^{\frac{2}{N}}+\cdots+2\delta_j
         \vspace{3mm}\\
         &\le& \displaystyle 1-\frac{\nu}{2}
    \end{array}
\end{equation}
and then
\begin{equation}\nonumber
    \int_0^{(1-\nu)a_j} \big(\frac{\log\frac{e}{s}}{s}\big)^{\frac{|v_j(s)|^N}{s^{N-1}}}\,ds\le \int_0^{(1-\nu)a_j} \big(\frac{\log\frac{e}{s}}{s}\big)^{1-\frac{\nu}{2}}\,ds\le c.
\end{equation}

(ii) The intervals $I_j^0:=\big[a_j(1-c\delta_j), a_j\big]$. Using \eqref{eq:87}, we can first estimate,
\begin{equation}\nonumber
\begin{array}{lcl}
     J_j^0&=&\displaystyle \int^{a_j}_{a_j(1-c\delta_j)} \Big(\frac{\log\frac{e}{s}}{s}\Big)^{\frac{|v_j(s)|^N}{s^{N-1}}}ds
     \vspace{3mm}\\
     &\le& \displaystyle \int_{a_j(1-c\delta_j)}^{a_j} e^{\log (\frac{\log\frac{e}{s}}{s})^{(1-\delta_j)}}\, ds
     \vspace{3mm}\\
     &\le& \displaystyle ca_j\delta_je^{(\log \frac{2}{a_j}+\log\log\frac{2e}{a_j})(1-\delta_j)}
     \vspace{3mm}\\
     &\le& \displaystyle ca_j\delta_j\frac{2}{a_j} \log(\frac{2e}{a_j}) e^{-\delta_j\log \frac{2}{a_j}}
     \vspace{3mm}\\
     &=& \displaystyle c\,2\,\frac{\log\frac{2e}{a_j}}{\log\frac{2}{a_j}}\, \delta_j\log\frac{2}{a_j}\ e^{-\delta_j\log \frac{2}{a_j}}
     \vspace{3mm}\\
     &\le&\displaystyle 3c
\end{array}
\end{equation}
for $j$ large, because $\lim\limits_{j\to \infty}\frac{\log(\frac{2e}{a_j})}{\log(\frac{2}{a_j})}=1$ and $xe^{-x}$ becomes maximal for $x=1$.
\par \medskip

(iii) $\bigcup_{k=1}^{k_j(\nu)}I_j^k$. Before estimating the integral on this interval, we first prove the following lemma.
\begin{lemma}\label{l4.2}
     Let $c\ge 2N^2+1 $, then

\begin{equation}\nonumber
    \frac{|v_j(s)|^N}{s^{N-1}}\le 1-k\delta_j, \ \ \forall s\in I_j^k,\ \, k=1,2,\dots, k_j(\nu)-1.
\end{equation}
\end{lemma}
\begin{proof}
    In fact, using Lemma\,\ref{Key2}, we can bound $v_j(s)$, $s\in I_j^k$, as follows
\begin{equation}\label{eq:91}\nonumber
    v_j(s)\le s\,\frac{v_j(a_j)}{a_j}+(a_j-s)^{\frac{1}{2}}\sqrt{\delta_j}\big(\frac{v_j(a_j)}{a_j}\big)^{-\frac{N-2}{2}},\ \ s\le a_j.
\end{equation}
Since
\begin{equation}\label{eq:92}
    \frac{|v_j(a_j)|^N}{a^{N-1}_j}=1-\delta_j\ \ \hbox{implies}\ \frac{v_j(a_j)}{a_j}=(1-\delta_j)^{\frac{1}{N}}a^{-\frac{1}{N}}_j
\end{equation}
we get
\begin{equation}\nonumber
    v_j(s)\le s(1-\delta_j)^{\frac{1}{N}}a^{-\frac{1}{N}}_j+(a_j-s)^{\frac{1}{2}}\sqrt{\delta_j}(1-\delta_j)^{-\frac{N-2}{2N}}a^{\frac{N-2}{2N}}_j,\ \ s\le a_j.
\end{equation}
By direct calculations, one then has
\begin{equation}\label{eq:94}
    \begin{array}{lcl}
         \displaystyle \frac{|v_j(s)|^N}{s^{N-1}}
         &\le& \frac{1}{s^{N-1}} \Big|s(1-\delta_j)^{\frac{1}{N}}a^{-\frac{1}{N}}_j+(a_j-s)^{\frac{1}{2}}\sqrt{\delta_j}(1-\delta_j)^{-\frac{N-2}{2N}}a^{\frac{N-2}{2N}}_j\Big|^N
         \vspace{3mm}\\
         &=& \displaystyle \Big|\frac{s^{\frac{1}{N}}}{a^{\frac{1}{N}}_j}(1-\delta_j)^{\frac{1}{N}}+\big(\frac{a_j}{s^{\frac{2(N-1)}{N}}}-s^{\frac{2-N}{N}}\big)^{\frac{1}{2}}\sqrt{\delta_j}(1-\delta_j)^{-\frac{N-2}{2N}}a^{\frac{N-2}{2N}}_j\Big|^N
         \vspace{3mm}\\
          &=& \displaystyle \Big|\frac{s^{\frac{1}{N}}}{a^{\frac{1}{N}}_j}(1-\delta_j)^{\frac{1}{N}}+\frac{a^{\frac{N-2}{2N}}_j}{s^{\frac{N-2}{2N}}}(\frac{a_j}{s}-1)^{\frac{1}{2}}\sqrt{\delta_j}(1-\delta_j)^{-\frac{N-2}{2N}}\Big|^N.
         \vspace{3mm}\\
    \end{array}
\end{equation}
Based on this we further obtain the following estimate for $s\in I_j^k$,
\begin{equation}\nonumber
\aligned
     &\frac{|v_j(s)|^N}{s^{N-1}}
     \vspace{3mm}\\ \le &\bigg|\frac{\big(a_j(1-kc\delta_j)\big)^{\frac{1}{N}}(1-\delta_j)^{\frac{1}{N}}}{a^{\frac{1}{N}}_j}+\frac{a^{\frac{N-2}{2N}}_j \sqrt{\delta_j}(1-\delta_j)^{-\frac{N-2}{2N}}}{(a_j(1-(k+1)c\delta_j))^{\frac{N-2}{2N}}}\left(\frac{a_j}{a_j(1-(k+1)c\delta_j)}-1\right)^{\frac{1}{2}}\bigg|^N
     \vspace{3mm}\\
     =&\bigg|(1-kc\delta_j)^{\frac{1}{N}}(1-\delta_j)^{\frac{1}{N}} +\frac{1}{(1-(k+1)c\delta_j)^{\frac{N-2}{2N}}}\left(\frac{1}{1-(k+1)c\delta_j}-1\right)^{\frac{1}{2}}\sqrt{\delta_j}(1-\delta_j)^{-\frac{N-2}{2N}}\bigg|^N.
\endaligned
\end{equation}
We now select a fixed small $\nu>0$, and note that $k_j(\nu)\in \mathbb{N}$ denote the smallest integer such that $k_j(\nu)c\delta_j\ge \nu$; then $O((kc\delta_j)^2)=o(kc\delta_j)$, for all $k=1,2,\dots, k_j(\nu)-1$.

By employing the Taylor series expansion, we then continue the estimate
\begin{equation}\nonumber
    \begin{array}{lcl}
         \displaystyle \frac{|v_j(s)|^N}{s^{N-1}}
     &\le& \displaystyle \bigg|\Big(1-\frac{1}{N}kc\delta_j+O((kc\delta_j)^2)\Big)\big(1-\frac{1}{N}\delta_j+O(\delta^2_j) \big) \vspace{3mm}\\
&&\displaystyle+\Big(1+\hbox{$\frac{N-2}{2N}$}(k+1)c\delta_j+O\big(((k+1)c\delta_j)^2\big) \Big)
     \vspace{3mm}\\
     &&\displaystyle \cdot\Big(1+(k+1)c\delta_j+O\big(((k+1)c\delta_j)^2\big)-1\Big)^\frac{1}{2}\delta^{\frac{1}{2}}_j\Big(1+\hbox{$\frac{N-2}{2N}$}\delta_j+O((\delta_j)^2)\Big)\bigg|^N

     \vspace{3mm}\\
     &=&\displaystyle \bigg|1-\frac{1}{N} \delta_j(kc+1)+O((kc\delta_j)^2)+\Big(1+\hbox{$\frac{N-2}{2N}$}(k+1)c\delta_j+O\big(((k+1)c\delta_j)^2\big)\Big)
     \vspace{3mm}\\
     &&\displaystyle\cdot\Big((k+1)c\delta_j+O\big(((k+1)c\delta_j)^2\big)\Big)^\frac{1}{2}\delta^{\frac{1}{2}}_j\Big(1+\hbox{$\frac{N-2}{2N}$}\delta_j+O((\delta_j)^2)\Big)\bigg|^N
     \vspace{3mm}\\
     &=&\displaystyle \Big|1-\frac{1}{N}  \delta_j(kc+1)+O((kc\delta_j)^2)+((k+1)c)^{\frac{1}{2}}\delta_j\Big|^N
     \vspace{3mm}\\
     &=&\displaystyle \Big|1-\delta_j\frac{1}{N}\Big((kc+1)-N((k+1)c)^{\frac{1}{2}}\Big)+O((kc\delta_j)^2\Big|^N
     \vspace{3mm}\\
     &=&\displaystyle 1-\delta_j\Big((kc+1)-N((k+1)c)^{\frac{1}{2}}\Big)+O((kc\delta_j)^2.
    \end{array}
\end{equation}
Ultimately, it can be verified that
\begin{equation}\nonumber
    (kc+1)-N((k+1)c)^{\frac{1}{2}}>k, \ \hbox{for}\ c\ge 2N^2+1
\end{equation}
and hence the lemma holds for all $k=1,2,\dots, k_j(\nu)-1$ and $j$ sufficiently large.
\end{proof}

\par \medskip
Now, we can consider the integrals over the intervals $I_j^k=\big[a_j(1-(k+1)c\delta_j), a_j(1-kc\delta_j)\big]$. By Lemma\,\ref{l4.2}, we get
\begin{equation*}
\begin{array}{lcl}
         \displaystyle
J_j^k:&=&\displaystyle \int^{a_j(1-kc\delta_j)}_{a_j(1-(k+1)c\delta_j)} \big(\frac{\log\frac{e}{s}}{s}\big)^{\frac{|v_j(s)|^N}{s^{N-1}}}\, ds
     \le \displaystyle \int^{a_j(1-kc\delta_j)}_{a_j(1-(k+1)c\delta_j)} \big(\frac{\log\frac{e}{s}}{s}\big)^{(1-k\delta_j)}\, ds
      \vspace{3mm}\\
     &\le& \displaystyle \int^{a_j(1-kc\delta_j)}_{a_j(1-(k+1)c\delta_j)} e^{(\log \frac{2}{a_j}+\log\log \frac{2e}{a_j})(1-k\delta_j)}\, ds
     =\displaystyle ca_j\delta_j e^{(\log \frac{2}{a_j}+\log\log \frac{2e}{a_j})(1-k\delta_j)}
     \vspace{3mm}\\
     &\le& \displaystyle ca_j\delta_j \frac{2}{a_j}\log\frac{2}{a_j}\, e^{-k\delta_j\log \frac{2}{a_j}}
    = \displaystyle 2c\,\delta_j \log \frac{2}{a_j}\, e^{-k\delta_j\log \frac{2}{a_j}}
      \vspace{3mm}\\
     &\le& \displaystyle 2c\,\frac{d_j}{e^{kd_j}} \ \hbox{with}\ d_j:=\delta_j\log\frac{2}{a_j}.
     \end{array}
\end{equation*}
In conclusion, we derive
$$
  \int^{a_j}_{(1-\nu)a_j}  \big(\frac{\log\frac{e}{s}}{s}\big)^{\frac{|v_j(s)|^N}{s^{N-1}}}\, ds\le  \sum^{k_j(\nu)}_{k=0} J_j^k\le 2c\sum_{k\ge 1}\frac{d_j}{e^{kd_j}} .
$$
Applying the ratio test, we get by $d_j\geq\delta_j\log\frac{1}{a_j}\ge d_0$ that
$$\frac{\frac{d_j}{e^{(k+1)d_j}}}{\frac{d_j}{e^{kd_j}}}=\frac{1}{e^{d_j}}\leq\frac{1}{e^{d_0}}, $$
and hence the infinite series $\sum_{k\ge 1}\frac{d_j}{e^{kd_j}}$ converges uniformly with respect to $j$.

\vspace{1mm}\par \medskip \noindent
{\it Case 2:} $\delta_j\log\frac{1}{a_j}\to 0$.
\par \noindent
In this case we choose another fixed small $\nu>0$, and  consider:\\
(i) the left sub-interval $(0, (1-\nu)a_j)$; \\
(ii) the right part of the interval $((1-\nu)a_j, a_j)\subset\bigcup_{k=1}^{k_j(\nu)}I_j^k$, where $$I_j^k=\big[a_j(1-\frac{k}{\log\frac{1}{a_j}}), a_j(1-\frac{k-1}{\log\frac{1}{a_j}})\big], \,\ k=1,2,\cdots, k_j(\nu),$$
and
$k_j(\nu)\in \mathbb{N}$ denote the smallest integer such that $\frac{k_j (\nu)}{\log\frac{1}{a_j}}\ge \nu$.\\
For the interval $(0, (1-\nu)a_j)$, reapplying \eqref{eq:15} we have
\begin{equation}\nonumber
    \int^{(1-\nu)a_j}_{0}  \Big(\frac{\log\frac{e}{s}}{s}\Big)^{\frac{v^N_j(s)}{s^{N-1}}}\, ds\le  \int^{(1-\nu)a_j}_{0}  \Big(\frac{\log\frac{e}{s}}{s}\Big)^{1-\frac{\nu}{2}}\, ds\le C.
\end{equation}
For the interval $\bigcup_{k=1}^{k_j(\nu)}I_j^k\subset [\frac{a_j}{2}, a_j]$, equation \eqref{eq:94} yields the estimate
\begin{equation}\nonumber
    \begin{array}{lcl}
        \displaystyle \frac{v^N_j(s)}{s^{N-1}}
        &\le& \displaystyle \Big|\frac{s^{\frac{1}{N}}}{a^{\frac{1}{N}}_j}(1-\delta_j)^{\frac{1}{N}}+\frac{a^{\frac{N-2}{2N}}_j}{s^{\frac{N-2}{2N}}}(\frac{a_j}{s}-1)^{\frac{1}{2}}\sqrt{\delta_j}(1-\delta_j)^{-\frac{N-2}{2N}}\Big|^N
        \vspace{3mm}\\
        &=& \displaystyle \Big|\frac{s}{a_j}(1-\delta_j)+N\frac{s^{\frac{N-1}{N}}}{a^{\frac{N-1}{N}}_j}(1-\delta_j)^{\frac{N-1}{N}}\frac{a^{\frac{N-2}{2N}}_j}{s^{\frac{N-2}{2N}}}(\frac{a_j}{s}-1)^{\frac{1}{2}}\sqrt{\delta_j}(1-\delta_j)^{-\frac{N-2}{2N}}+\cdots
        \vspace{3mm}\\
        &&+N\frac{s^{\frac{1}{N}}}{a^{\frac{1}{N}}_j}(1-\delta_j)^{\frac{1}{N}}\frac{a^{\frac{(N-2)(N-1)}{2N}}_j}{s^{\frac{(N-2)(N-1)}{2N}}}(\frac{a_j}{s}-1)^{\frac{N-1}{2}}\delta_j^{\frac{N-1}{2}}(1-\delta_j)^{-\frac{(N-2)(N-1)}{2N}}
        \vspace{3mm}\\
        &&+\frac{a^{\frac{N-2}{2}}_j}{s^{\frac{N-2}{2}}}(\frac{a_j}{s}-1)^{\frac{N}{2}}\delta_j^{\frac{N}{2}}(1-\delta_j)^{-\frac{N-2}{2}}\Big|.
    \end{array}
\end{equation}
Therefore, for $s\in I_j^k,~ k=1,2,\cdots, k_j(\nu)$, applying binomial theorem we find
\begin{equation}\nonumber
    \begin{array}{lcl}
         \displaystyle \frac{v^N_j(s)}{s^{N-1}}
        &\le& \displaystyle (1-\frac{k-1}{\log\frac{1}{a_j}})+\frac{N}{\sqrt{a_j}}\sqrt{\delta_j}\big(\frac{a_jk}{\log\frac{1}{a_j}}\big)^{\frac{1}{2}}(1-\delta_j)^{-\frac{N-2}{2N}}+\cdots
        \vspace{3mm}\\
        &&+N\frac{(\frac{k}{\log\frac{1}{a_j}})^{\frac{N-1}{2}}}{(1-\frac{k}{\log\frac{1}{a_j}})^{N-2}}\delta_j^{\frac{N-1}{2}}(1-\delta_j)^{-\frac{(N-2)(N-1)}{2N}}+\frac{(\frac{k}{\log\frac{1}{a_j}})^{\frac{N}{2}}\delta^{\frac{N}{2}}_j}{(1-\frac{k}{\log\frac{1}{a_j}})^\frac{2(N-1)}{2}}(1-\delta_j)^{-\frac{N-2}{2}}
         \vspace{3mm}\\
         &= & \displaystyle 1-\frac{k-1}{\log\frac{1}{a_j}}+o(\frac{k}{\log\frac{1}{a_j}})
          \vspace{3mm}\\
         &\le & \displaystyle 1-\frac{1}{2}\frac{k-1}{\log\frac{1}{a_j}}\ \ \ \hbox{for $j$ sufficient large}.
    \end{array}
\end{equation}
Thus, we see that for the integrals over $I_j^k\subset [\frac{a_j}{2},a_j]$
\begin{equation}\nonumber
    \begin{array}{lcl}
         \displaystyle J_j^k:&=&\displaystyle \int_{a_j(1-\frac{k}{\log\frac{1}{a_j}})}^{a_j(1-\frac{k-1}{\log\frac{1}{a_j}})} \Big(\frac{\log\frac{e}{s}}{s}\Big)^{\frac{v^N_j(s)}{s^{N-1}}}\,ds
         \vspace{3mm}\\
         &\le& \displaystyle \int_{a_j(1-\frac{k}{\log\frac{1}{a_j}})}^{a_j(1-\frac{k-1}{\log\frac{1}{a_j}})}\ e^{(\log\frac{2}{a_j}+\log\log\frac{2e}{a_j})(1-\frac{1}{2}\frac{k-1}{\log\frac{1}{a_j}})}\, ds
         \vspace{3mm}\\
         &\le& \displaystyle a_j\frac{1}{\log\frac{1}{a_j}}\frac{2}{a_j}\log\big(\frac{2e}{a_j}\big)e^{-\frac{k-1}{2\log\frac{1}{a_j}}(\log\frac{2}{a_j}+\log\log\frac{2e}{a_j})}
         \vspace{3mm}\\
         &\le& \displaystyle 3\big(\frac{1}{\sqrt{e}}\big)^{\frac{k-1}{2}}\ \ \hbox{for $j$ sufficient large} \ ,
    \end{array}
\end{equation}
given that $$\frac{1}{\log\frac{1}{a_j}}\log(\frac{2e}{a_j})\to 1~~\text{and}~~\frac{k-1}{2\log\frac{1}{a_j}}(\log\frac{2}{a_j}+\log\log\frac{2e}{a_j})\to \frac{k-1}{2}.$$

Notice that $k_j(\nu)$ is the smallest integer such that $\frac{k_j (\nu)}{\log\frac{1}{a_j}}\ge \nu$, $\nu > 0$ small. Hence, it follows
again that
\begin{equation}\nonumber
    \int_{(1-\nu)a_j}^{a_j}  \Big(\frac{\log\frac{e}{s}}{s}\Big)^{\frac{v^N_j(s)}{s^{N-1}}}\, ds\le \sum_{k=1}^{k_j(\nu)} J_j^k\le 3\sum_{k\ge 1} \big(\frac{1}{e^{1/4}}\big)^{k-1} \le C.
\end{equation}

\par \medskip \indent
{\bf 2. The interval $(a_j, 1):$ }
\par \smallskip \noindent
{\it Case 1}: there exists $d_0 >0$ such that $\delta_j \log\frac{1}{a_j}\ge d_0$, for all $j$.\\
Likewise, we now we choose another fixed small $\nu>0$ and examine \\
(i) the left sub-interval $I_j^0:=\big[a_j, a_j(1+c\delta_j)\big]$; \\
(ii) the right sub-interval $\big[b_j, 1\big]$, $b_j=a_j(1+k_j(\nu)c\delta_j)$, where $k_j(\nu)\in \mathbb{N}$ denote the smallest integer such that $k_j(\nu)c\delta_j\ge \nu$;
 \\
(iii) the middle part $\bigcup_{k=1}^{k_j(\nu)-1}I_j^k$,   $$
    I_j^k=\big[a_j(1+kc\delta_j), a_j(1+(k+1)c\delta_j)\big] \quad k=1,2\cdots, k_j(\nu)-1.
$$
\par \smallskip \indent
(i) The left sub-interval $I_j^0:=\big[a_j, a_j(1+c\delta_j)\big]$. Firstly, on the interval $[a_j, 1]$, it follows from Lemma\,\ref{Key1} that
\begin{equation}\label{eq:102}\nonumber
    w(s)\leq w(a_j)+(s-a_j)^{1-\frac{1}{N}}\delta^{\frac{1}{N}},\quad\forall s\in[a_j,1],
\end{equation}
which leads to $v_j$ by \eqref{eq:92},
\begin{equation}\nonumber
    v_j(s)\le a_j^{\frac{N-1}{N}}(1-\delta_j)^{\frac{1}{N}}+(s-a_j)^{1-\frac{1}{N}}\delta^{\frac{1}{N}},\ \ a_j\le s\le 1.
\end{equation}
For $s\in [a_j, a_j(1 + c\delta_j)]$ we have by definition
\begin{equation}\nonumber
    \frac{|v_j(s)|^N}{s^{N-1}}\le 1-\delta_j
\end{equation}
and hence
\begin{equation}\nonumber
\begin{array}{lcl}
     \displaystyle K_j^0&:=&\displaystyle \int^{a_j(1 + c\delta_j)}_{a_j} \Big(\frac{\log\frac{e}{s}}{s}\Big)^{\frac{|v_j(s)|^N}{s^{N-1}}}\,ds\vspace{3mm}\\
     &\le& \displaystyle \int^{a_j(1 + c\delta_j)}_{a_j} e^{(\log \frac{1}{s}+\log\log\frac{e}{s})(1-\delta_j)}\, ds
     \vspace{3mm}\\
     &\le&\displaystyle  \int^{a_j(1 + c\delta_j)}_{a_j} e^{(\log \frac{1}{a_j}+\log\log\frac{e}{a_j})(1-\delta_j)}\,ds
     \vspace{3mm}\\
     &=&\displaystyle a_jc\delta_j e^{(\log \frac{1}{a_j}+\log\log\frac{e}{a_j})}e^{(\log \frac{1}{a_j}+\log\log\frac{e}{a_j})-\delta_j}
     \vspace{3mm}\\
     &=&\displaystyle a_jc\delta_j\frac{1}{a_j}\log\frac{e}{a_j}e^{(\log \frac{1}{a_j}+\log\log\frac{e}{a_j})-\delta_j}
     \vspace{3mm}\\
     &\le&\displaystyle c\delta_j\log\frac{e}{a_j}e^{\log\frac{1}{a_j}(-\delta_j)}
      \vspace{3mm}\\
     &=&\displaystyle c\frac{\delta_j\log\frac{e}{a_j}}{e^{\delta_j\log\frac{1}{a_j}}}\le c.
\end{array}
\end{equation}
(ii) Before we discuss the interval $[a_j(1 + k_j(\nu)c\delta_j, 1]$, we introduce a lemma similar to Lemma \ref{l4.2}. Lemma \ref{l4.3} will be also used in (iii) the interval $I_j^k$.
\par \smallskip \noindent
\begin{lemma}\label{l4.3}
    If $(c-1)^{\frac{N}{N-1}}>2N^{\frac{N}{N-1}}c$, then for $s\in I_j^k, k=1,2,\dots,k_j(\nu)-1$ we have
\begin{equation}\nonumber
    \frac{|v_j(s)|^N}{s^{N-1}}\le 1-k\delta_j.
\end{equation}
\end{lemma}
\begin{proof}
    From Lemma\,\ref{Key1},
\begin{equation}\nonumber
    v_j(s)\le v(a_j)+(s-a_j)^{1-\frac{1}{N}}\delta^{\frac{1}{N}},\,\ \forall s\in [a_j, 1]
\end{equation}
we calculate for $s\in \big(a_j(1+kc\delta_j), a_j(1+(k+1)c\delta_j)\big)$

\begin{equation}\nonumber
\begin{array}{lcl}
     v_j(s)
     &\le& a_j^{\frac{N-1}{N}}(1-\delta_j)^{\frac{1}{N}}+\big(a_j(k+1)c\delta_j\big)^{1-\frac{1}{N}}\delta^{\frac{1}{N}}_j
     \vspace{3mm}\\
     &\le& \displaystyle a_j^{\frac{N-1}{N}}\Big((1-\delta_j)^{\frac{1}{N}}+\big((k+1)c\delta_j\big)^{1-\frac{1}{N}}\delta^{\frac{1}{N}}_j \Big)
     \vspace{3mm}\\
     &\le& \displaystyle a_j^{\frac{N-1}{N}} \Big(1+\big((k+1)c\delta_j\big)^{1-\frac{1}{N}}\delta^{\frac{1}{N}}_j\Big).
\end{array}
\end{equation}
Similarly, we can estimate for $s\in \big(a_j(1+kc\delta_j), a_j(1+(k+1)c\delta_j)\big)$
\begin{equation}\nonumber
    \begin{array}{lcl}
         \displaystyle \frac{|v_j(s)|^N}{s^{N-1}}
         &\le&  \displaystyle \frac{\Big|a_j^{\frac{N-1}{N}} \Big(1+((k+1)c\delta_j)^{1-\frac{1}{N}}\delta^{\frac{1}{N}}_j\Big)\Big|^N}{\big(a_j(1+kc\delta_j)\big)^{N-1}}
         \vspace{3mm}\\
     &\le& \displaystyle \frac{ 1+N\big((k+1)c\big)^{1-\frac{1}{N}}\delta_j}{(1+kc\delta_j)^{N-1}}+O\Big(\big((kc)^{\frac{N-1}{N}}\delta_j\big)^2\Big)
     \vspace{3mm}\\
     &=& \displaystyle  \frac{ 1+N\big((k+1)c\big)^{1-\frac{1}{N}}\delta_j+(kc\delta_j)-(kc\delta_j)}{(1+kc\delta_j)^{N-1}}+O\Big(\big((kc)^{\frac{N-1}{N}}\delta_j\big)^2\Big)
      \vspace{3mm}\\
     &\le& \displaystyle 1+\frac{ N\big((k+1)c\big)^{1-\frac{1}{N}}\delta_j-(kc\delta_j)}{(1+kc\delta_j)^{N-1}}+O\Big(\big((kc)^{\frac{N-1}{N}}\delta_j\big)^2\Big)
     \vspace{3mm}\\
     &\le& \displaystyle 1-\Big((kc)-N\big((k+1)c\big)^{1-\frac{1}{N}}\Big)\delta_j+O\Big(\big((kc)^{\frac{N-1}{N}}\delta_j\big)^2\Big).
    \end{array}
\end{equation}
Setting $\delta_j^k:= \Big(kc-N\big((k+1)c\big)^{1-\frac{1}{N}}\Big)\delta_j+O\Big(\big((kc)^{\frac{N-1}{N}}\delta_j\big)^2\Big)$, we see that
\begin{equation}\label{eq:111}\nonumber
    \delta_j^k>k\delta_j,\quad\hbox{if}\ (c-1)^{\frac{N}{N-1}}>2N^{\frac{N}{N-1}}c\ \hbox{and}\ j \ \hbox{sufficient large}
\end{equation}
and hence
\begin{equation}\nonumber
    \displaystyle \frac{|v_j(s)|^N}{s^{N-1}}\le 1-\delta_j^k\le 1-k\delta_j.
\end{equation}
\end{proof}

\noindent Now let us denote
\begin{equation}\nonumber
    \begin{array}{lcl}
         k_j:=k_j(\nu)\ \ \ \hbox{and}\ \ b_j=a_j(1+k_jc\delta_j)
    \end{array}
\end{equation}
and note that by the inequality
\begin{equation}\nonumber
\begin{array}{lcl}
     v_j(s)&=&\displaystyle v_j(b_j)+\int_{b_j}^s v'_j(s)\,ds
     \vspace{3mm}\\
     &\le& \displaystyle v_j(b_j)+(s-b_j)^{\frac{N-1}{N}}\big(\int_{b_j}^s|v'_j(s)|^N\,ds\big)^{\frac{1}{N}}
     \vspace{3mm}\\
     &\le& \displaystyle v_j(b_j)+(s-b_j)^{\frac{N-1}{N}}\delta_j^{\frac{1}{N}},\ \ s\in (b_j,1].
\end{array}
\end{equation}
On the other hand, by Lemma\,\ref{l4.3}, we know that
\begin{equation}\nonumber
   \frac{|v_j(b_j)|^N}{b_j^{N-1}}\le 1-k_j\delta_j\le 1-\frac{\nu}{c},\ \ \ s\in I_j^k
\end{equation}
which yields
\begin{equation}\nonumber
    v_j(b_j)\le b^{\frac{N-1}{N}}_j(1-\frac{\nu}{c})^{\frac{1}{N}}
\end{equation}
subsequently, following from the above
\begin{equation}\nonumber
    v_j(s)\le  b^{\frac{N-1}{N}}_j(1-\frac{\nu}{c})^{\frac{1}{N}}+(s-b_j)^{\frac{N-1}{N}}\delta_j^{\frac{1}{N}},
\end{equation}
which in turn implies
\begin{equation}\nonumber
    \begin{array}{lcl}
         \displaystyle \frac{|v_j(s)|^N}{s^{N-1}}&\le& \displaystyle \frac{\big| b^{\frac{N-1}{N}}_j(1-\frac{\nu}{c})^{\frac{1}{N}}+(s-b_j)^{\frac{N-1}{N}}\delta_j^{\frac{1}{N}}\big|^N}{s^{N-1}}
         \vspace{3mm}\\
     &\le& \displaystyle \big|(1-\frac{\nu}{c})^{\frac{1}{N}}+\delta_j^{\frac{1}{N}}\big|^N
      \vspace{3mm}\\
     &\le& \displaystyle 1-\frac{\nu}{2c}\ \ \hbox{for $j$ sufficiently large.}
    \end{array}
\end{equation}
Then
\begin{equation}\nonumber
    \displaystyle \int_{b_j}^1 \Big(\frac{\log\frac{e}{s}}{s}\Big)^{\frac{|v_j(s)|^N}{s^{N-1}}}\,ds\le \int_{b_j}^1 \Big(\frac{\log\frac{e}{s}}{s}\Big)^{1-\frac{\nu}{2c}} \,ds\le c\ \ \hbox{for all $v\in E_N$}.
\end{equation}

(iii) Next, let us consider the intervals
\begin{equation}\nonumber
    I_j^k=\big[a_j(1+kc\delta_j), a_j(1+(k+1)c\delta_j)\big] \quad k=1,2,\dots,k_j(\nu)-1.
\end{equation}

By Lemma \ref{l4.3}, it is easy to check on the intervals $I_j^k$,
\begin{equation}\nonumber
  \begin{array}{lcl}
 J_j^k:&=&\displaystyle \int_{a_j(1+kc\delta_j)}^{a_j(1+(k+1)c\delta_j)} \Big(\frac{\log\frac{e}{s}}{s}\Big)^{\frac{|v_j(s)|^N}{s^{N-1}}}\,ds
 \vspace{3mm}\\
 &\le &\displaystyle \int_{a_j(1+kc\delta_j)}^{a_j(1+(k+1)c\delta_j)} e^{(\log\frac{1}{a_j}+\log\log\frac{e}{a_j})(1-k\delta_j)}\,ds
 \vspace{3mm}\\
 &\le &\displaystyle a_jc\delta_j \frac{1}{a_j}\log\big(\frac{e}{a_j}\big)e^{-k\delta_j\log\frac{1}{a_j}}
 \vspace{3mm}\\
 &\le&\displaystyle 2c\delta_j \log\big(\frac{1}{a_j}\big)e^{-k\delta_j\log\frac{1}{a_j}}
 \vspace{3mm}\\
 &= &\displaystyle \frac{2cd_j}{e^{kd_j}},
  \end{array}
\end{equation}
where $d_j:=\delta_j\log\frac{1}{a_j}$.

Then by the assumption $d_j=\delta_j\log\frac{1}{a_j}\ge d_0$ we obtain that
\begin{equation}\nonumber
    \begin{array}{lcl}
         \displaystyle \int_{a_j}^{a_j(1+k_j(\nu)c\delta_j)} \Big(\frac{\log\frac{e}{s}}{s}\Big)^{\frac{|v_j(s)|^N}{s^{N-1}}}\,ds&\le& \displaystyle \sum_{k=0}^{k_j(\nu)} J_j^k
         \vspace{3mm}\\
 &\le &\displaystyle 2c\sum_{k\ge0} \frac{d_j}{e^{kd_j}} \le c.
    \end{array}
\end{equation}

\par \medskip
\noindent{\it Case 2:} $\delta_j\log\frac{1}{a_j}\to 0$, that is $\delta_j=o(\frac{1}{\log(\frac{1}{a_j})})$.
\par \smallskip \noindent
Let us deal with now \\
(i) the left sub-interval $\big[a_j, b_j\big]= \bigcup_{k=0}^{k_j(\nu)-1}I_j^k$, $b_j=a_j\Big(1+\frac{k_j(\nu)}{\log\frac{1}{a_j}}\Big)$, where $k_j(\nu)$ denote the smallest integer with $\frac{k_j(\nu)}{\log\frac{1}{a_j}}\ge \nu$, for some small $\nu>0$ and
\begin{equation}\nonumber
I_j^k=\big[a_j\big(1+\frac{k}{\log\frac{1}{a_j}}\big), a_j\big(1+\frac{k+1}{\log\frac{1}{a_j}}\big)\big], \,\ k=0,1,2,\cdots,k_j(\nu)-1.
\end{equation}
(ii) the right sub-interval $\big[b_j, 1\big]$.

(i) For $s\in I_j^k$, we have for $j$ sufficiently large
\begin{equation}\nonumber
    \frac{v^N(s)}{s^{N-1}}\le 1-\frac{1}{2}\frac{k}{\log\frac{1}{a_j}}.
\end{equation}
Indeed, we know from Lemma\,\ref{Key} and Lemma\,\ref{Key1} that for $a_j\le s\le a_j(1+\frac{k+1}{\log\frac{1}{a_j}})$
\begin{equation}\nonumber
    \begin{array}{lcl}
         v_j(s)&\le& v(a_j)+(s-a_j)^{1-\frac{1}{N}}\delta_j^{\frac{1}{N}}
         \vspace{3mm}\\
         &\le& a_j^{\frac{N-1}{N}}(1-\delta_j)^{\frac{1}{N}}+a_j^{\frac{N-1}{N}}(\frac{k+1}{\log\frac{1}{a_j}})^{\frac{N-1}{N}}\delta_j^{\frac{1}{N}},
    \end{array}
\end{equation}
which suggests that for $s\in \big[a_j(1+\frac{k}{\log\frac{1}{a_j}}), a_j(1+\frac{k+1}{\log\frac{1}{a_j}})\big]$
\begin{equation}\label{202516-e1}
    \begin{array}{lcl}
         \displaystyle \frac{v^N_j(s)}{s^{N-1}}&\le& \displaystyle
         \frac{\Big|a_j^{\frac{N-1}{N}}(1-\delta_j)^{\frac{1}{N}}+a_j^{\frac{N-1}{N}}\big(\frac{k+1}{\log\frac{1}{a_j}}\big)^{\frac{N-1}{N}}\delta_j^{\frac{1}{N}}\Big|^N}{a_j^{N-1}(1+\frac{k}{\log\frac{1}{a_j}})^{N-1}}
         \vspace{3mm}\\
         &=& \displaystyle \displaystyle
         \frac{\Big|(1-\delta_j)^{\frac{1}{N}}+\big(\frac{k+1}{\log\frac{1}{a_j}}\big)^{\frac{N-1}{N}}\delta_j^{\frac{1}{N}}\Big|^N}{(1+\frac{k}{\log\frac{1}{a_j}})^{N-1}}
         \vspace{3mm}\\
         &=& \displaystyle \frac{(1-\delta_j)+N(1-\delta_j)^{\frac{N-1}{N}}\big(\frac{k+1}{\log\frac{1}{a_j}}\big)^{\frac{N-1}{N}}\delta_j^{\frac{1}{N}}+\cdots+(\frac{k+1}{\log\frac{1}{a_j}})^{N-1}\delta_j}{(1+\frac{k}{\log\frac{1}{a_j}})^{N-1}}
         \vspace{3mm}\\
         &\le & \displaystyle 1-\frac{k}{\log\frac{1}{a_j}}+o\Big(\frac{(k+1)^{\frac{N-1}{N}}}{\log\frac{1}{a_j}}\Big)+o\Big(\frac{(k+1)^{\frac{2(N-1)}{N}}}{\log\frac{1}{a_j}}\Big)+\cdots+o\Big(\frac{(k+1)^{N-1}}{\log\frac{1}{a_j}}\Big)
          \vspace{3mm}\\
         &\le & \displaystyle 1-\frac{1}{2}\frac{k}{\log\frac{1}{a_j}}\ \ \ \hbox{for $j$ sufficient large}.
    \end{array}
\end{equation}
We then derive the estimates for the integrals over the intervals
$$
I_j^k=\big[a_j(1+\frac{k}{\log\frac{1}{a_j}}), a_j(1+\frac{k+1}{\log\frac{1}{a_j}})\big]\ , \  k=0,1,2, \dots,k_j(\nu)-1 $$
that
\begin{equation}\nonumber
    \begin{array}{lcl}
         \displaystyle K_j^k:&=&\displaystyle \int_{a_j(1+\frac{k}{\log\frac{1}{a_j}})}^{a_j(1+\frac{k+1}{\log\frac{1}{a_j}})} (\frac{\log\frac{e}{s}}{s})^{\frac{v^N_j(s)}{s^{N-1}}}\,ds
         \vspace{3mm}\\
         &\le& \displaystyle \int_{a_j(1+\frac{k}{\log\frac{1}{a_j}})}^{a_j(1+\frac{k+1}{\log\frac{1}{a_j}})} e^{(\log\frac{1}{a_j}+\log\log\frac{e}{a_j})(1-\frac{1}{2}\frac{k}{\log\frac{1}{a_j}})}\, ds
         \vspace{3mm}\\
         &\le& \displaystyle a_j\frac{1}{\log\frac{1}{a_j}}\ \frac{1}{a_j}\log(\frac{e}{a_j})e^{-\frac{k}{2}}
         \vspace{3mm}\\
         &\le& \displaystyle 2e^{-\frac{k}{2}}\ \ \hbox{for $j$ sufficient large}.
    \end{array}
\end{equation}
Then it follows again that
\begin{equation}\nonumber
    \int_{a_j}^{b_j}\Big(\frac{\log\frac{e}{s}}{s}\Big)^{\frac{|v_j(s)|^N}{s^{N-1}}}\le \sum_{k=0}^{k_j(\nu)}K_j^k\le 2\sum_{k\ge 0} e^{-\frac{k}{2}}\le c.
\end{equation}

(ii) For the interval $[b_j,1]$, we have for $b_j\le s\le1$,
\begin{equation}\nonumber
    \begin{array}{lcl}
         v_j(s)-v_j(b_j)&=&\displaystyle \int_{b_j}^s v'_j(t)\, dt
         \vspace{3mm}\\
         &\le& \displaystyle (s-b_j)^{\frac{N-1}{N}}\Big(\int_{b_j}^s|v'_j(t)|^N\, dt\Big)^{\frac{1}{N}}
          \vspace{3mm}\\
         &\le& \displaystyle (s-b_j)^{\frac{N-1}{N}}(\delta_j)^{\frac{1}{N}}.
    \end{array}
\end{equation}
Since by \eqref{202516-e1}
\begin{equation}\nonumber
    \frac{v^N_j(b_j)}{b^{N-1}_j}\le1-\frac{1}{2}\frac{k_j(\nu)}{\log\frac{1}{a_j}}\le 1-\frac{1}{2}\nu
\end{equation}
it follows that
\begin{equation}\nonumber
    v_j(s)\le v_j(b_j)+(s-b_j)^{\frac{N-1}{N}}(\delta_j)^{\frac{1}{N}}\le b_j^{\frac{N-1}{N}}(1-\frac{1}{2}\nu)^{\frac{1}{N}}+(s-b_j)^{\frac{N-1}{N}}(\delta_j)^{\frac{1}{N}}
\end{equation}
and then
\begin{equation}\nonumber
    \begin{array}{lcl}
         \displaystyle \frac{|v_j(s)|^N}{s^{N-1}}&\le& \displaystyle \frac{\big|b_j^{\frac{N-1}{N}}(1-\frac{1}{2}\nu)^{\frac{1}{N}}+(s-b_j)^{\frac{N-1}{N}}(\delta_j)^{\frac{1}{N}}\big|^N}{s^{N-1}}
         \vspace{3mm}\\
         &\le& \displaystyle \big|(1-\frac{1}{2}\nu)^{\frac{1}{N}}+\delta_j^{\frac{1}{N}}\big|^N\ \ \ \ s\in (b_j, 1)
         \vspace{3mm}\\
         &\le& \displaystyle 1-\frac{\nu}{4} \ ,\quad \hbox{for $j$ sufficiently large.}
    \end{array}
\end{equation}
Then, finally, we obtain
\begin{equation}\nonumber
    \int_{b_j}^1 \Big(\frac{\log\frac{e}{s}}{s}\Big)^{\frac{|v_j(s)|^N}{s^{N-1}}}\,ds\le\int_{b_j}^1 \Big(\frac{\log\frac{e}{s}}{s}\Big)^{1-\frac{\nu}{4}}\,ds\le c.
\end{equation}
This completes the proof of Theorem\,\ref{Th2}.
\end{proof}

\end{document}